\documentclass[11pt]{amsart}

\usepackage{titletoc}
\usepackage[utf8]{inputenc}
\usepackage[T1]{fontenc}
\usepackage{indentfirst}
\usepackage[margin=1in]{geometry} 
\usepackage{lipsum}
\usepackage{amsmath,amscd,amsthm,amssymb,mathrsfs,stmaryrd,color,tikz-cd,graphicx,tikz,mathtools,subfigure}
\tikzcdset{arrow style=tikz,diagrams={>=stealth}}
\usepackage{enumerate}
\usepackage[bookmarks, bookmarksdepth=1, colorlinks=true, linkcolor=blue, citecolor=red, urlcolor=blue,pagebackref]{hyperref}
\usepackage{subfiles}
\usepackage{xcolor}

\usetikzlibrary{matrix,arrows,decorations.pathmorphing,cd,topaths,calc,graphs}
\usepackage[all,cmtip]{xy}

\usepackage{relsize}

\newcommand{\abs}[1]{\lvert#1\rvert}
\newcommand{\norm}[1]{\lVert#1\rVert}

\renewcommand{\Re}{\operatorname{Re}}
\renewcommand{\Im}{\operatorname{Im}}

\makeatletter
\newcommand{\subalign}[1]{%
  \vcenter{%
    \Let@ \restore@math@cr \default@tag
    \baselineskip\fontdimen10 \scriptfont\tw@
    \advance\baselineskip\fontdimen12 \scriptfont\tw@
    \lineskip\thr@@\fontdimen8 \scriptfont\thr@@
    \lineskiplimit\lineskip
    \ialign{\hfil$\m@th\scriptstyle##$&$\m@th\scriptstyle{}##$\crcr
      #1\crcr
    }%
  }
}
\makeatother

\newcommand{\BC}{\mathbb{C}}

\newcommand{\BR}{\mathbb{R}}

\newcommand{\cA}{\mathcal{A}}

\newcommand{\cB}{\mathcal{B}}

\newcommand{\cC}{\mathcal{C}}

\newcommand{\bD}{\mathbf{D}}

\newcommand{\cE}{\mathcal{E}}

\newcommand{\sF}{\mathscr{F}}

\newcommand{\sH}{\mathscr{H}}

\newcommand{\cO}{\mathcal{O}}

\newcommand{\bP}{\mathbf{P}}
\newcommand{\cP}{\mathcal{P}}

\newcommand{\cS}{\mathcal{S}}

\newcommand{\cW}{\mathcal{W}}

\newcommand{\fa}{\mathfrak{a}}

\newcommand{\bc}{\mathbf{c}}

\newcommand{\fg}{\mathfrak{g}}

\newcommand{\fk}{\mathfrak{k}}

\newcommand{\fm}{\mathfrak{m}}

\newcommand{\fn}{\mathfrak{n}}

\newcommand{\fp}{\mathfrak{p}}

\newcommand{\vertiii}[1]{{\left\vert\kern-0.25ex\left\vert\kern-0.25ex\left\vert #1 
    \right\vert\kern-0.25ex\right\vert\kern-0.25ex\right\vert}}

\newcommand{\rom}[1]{\uppercase\expandafter{\romannumeral #1\relax}}
\renewcommand{\emptyset}{\varnothing}

\renewcommand{\tilde}[1]{\widetilde{#1}}

\DeclareMathOperator{\Vol}{Vol}

\DeclareMathOperator{\supp}{Supp}

\DeclareMathOperator{\rank}{rank}

\DeclareMathOperator{\Ad}{Ad}

\newcommand{\SO}{\operatorname{SO}}

\newcommand{\fsl}{\mathfrak{sl}}

\makeatletter
\newcommand{\colim@}[2]{%
  \vtop{\m@th\ialign{##\cr
    \hfil$#1\operator@font lim$\hfil\cr
    \noalign{\nointerlineskip\kern1.5\ex@}#2\cr
    \noalign{\nointerlineskip\kern-\ex@}\cr}}%
}
\newcommand{\Lim}{%
  \mathop{\mathpalette\varlim@{\leftarrowfill@\scriptscriptstyle}}\nmlimits@
}

\newcommand{\colim}{%
  \mathop{\mathpalette\varlim@{\rightarrowfill@\scriptscriptstyle}}\nmlimits@
}

\makeatother

\newtheorem{theorem}{Theorem}[section]

\newtheorem{proposition}[theorem]{Proposition}

\newtheorem{lemma}[theorem]{Lemma}

\theoremstyle{definition}

\theoremstyle{remark}
\newtheorem{remark}[theorem]{Remark}
\newtheorem*{rmk}{Remark}

\usepackage[new]{old-arrows}

\usepackage{amssymb}

\usepackage[utf8]{inputenc}

\title{Higher rank spectral projector estimates and a Stein--Tomas restriction theorem}
\author{Phillip Harris, Jiaqi Hou, and Xiaoqi Huang}
\address{Mathematisches Institut, Endenicher Allee 60, 53115 Bonn, Germany}
\email{harris@math.uni-bonn.de}
\address{Department of Mathematics, Louisiana State University, Baton Rouge, LA 70803, USA}
\email{jhou7@lsu.edu}
\address{Department of Mathematics, Louisiana State University, Baton Rouge, LA 70803, USA}
\email{xhuang49@lsu.edu}
\date{}

\usepackage{hyperref}
\numberwithin{equation}{section}

\begin{document}

\begin{abstract}    
Let $X$ be a compact locally symmetric space of noncompact type. 
We prove sharp upper bounds for the $L^2\to L^p$ operator norms of spectral projectors for joint eigenfunctions of the full ring of invariant differential operators on $X$. This improves the corresponding result of Marshall \cite{marshall2016p}, by removing a logarithmic power loss at the critical exponent.
Using the same method, we also prove the endpoint case of Mockenhaupt's $(L^2,L^p)$ Fourier restriction theorem \cite{Moc91}. 
\end{abstract}

\maketitle

\section{Introduction}

\subsection{Spectral projector}
Let $M$ be a compact $n$-dimensional Riemannian manifold without boundary, and let $\Delta$ be the Laplace--Beltrami operator associated with the metric on $M$.
We let $0=\lambda_0^2<\lambda^2_1\le \lambda^2_2\le \cdots$ denote the eigenvalues labeled with respect to the multiplicity of ${-\Delta}$, and $e_{\lambda_j}$ the associated $L^2$-normalized eigenfunctions, that is,
\begin{equation*}
    (\Delta_g+\lambda^2_j )e_{\lambda_j}=0, \quad \text{and } \quad \int_M |e_{\lambda_j}(x)|^2 \, dx=1.
\end{equation*}
Moreover, $\{e_{\lambda_j}\}$ forms an orthonormal basis of $L^2(M)$.
Sogge \cite{Sog88} (see also Avakumovi\'c \cite{Avakumovic} and Levitan \cite{Levitan} when $p = \infty$) proved the universal bound for $p\ge2$
\begin{equation}\label{eq: Sogge bound}
    \|e_\lambda\|_{L^p(M)}\lesssim \lambda^{\delta(n,p)}.
\end{equation}
Here $\delta(n,p)$ is given by
\begin{equation*}
            \delta(n,p) = \begin{cases}
                \tfrac{n-1}{4} - \tfrac{n-1}{2p} \;&\text{ for }\;2\leq p\leq \tfrac{2(n+1)}{n-1},\\
                \tfrac{n-1}{2} - \tfrac{n}{p} \;\;\;\;&\text{ for }\;\tfrac{2(n+1)}{n-1}\leq p\leq\infty.
            \end{cases}
\end{equation*}
The estimate \eqref{eq: Sogge bound} is seen to be sharp on the standard round sphere $S^n$ by taking $e_\lambda$ to be zonal spherical functions or highest weight spherical harmonics.

For any $\lambda\geq1$, we let $\bP_\lambda$ be the spectral projector around $\lambda$ of radius 1, defined by
\[
\bP_\lambda f = \sum_{|\lambda_j-\lambda|\leq1}\langle f,e_{\lambda_j}\rangle e_{\lambda_j},\qquad f\in L^2(M).
\]
Then Sogge's bound \eqref{eq: Sogge bound} follows from the following sharp spectral projector estimate for any $p\geq2$:
\[
    \|\bP_\lambda f\|_{L^p(M)}\lesssim\lambda^{\delta(n,p)} \|f\|_{L^2(M)},\qquad \text{ for any }f\in L^2(M).
\]

It is sometimes possible to improve the upper bound in \eqref{eq: Sogge bound} by considering joint eigenfunctions of extra operators. See, for example, \cite{Tac19}. In Sarnak's letter to Morawetz \cite{sarnak2004letter}, it is shown that if $X$ is a compact locally symmetric space of dimension $n$ and rank $r$, and $\psi$ is an $L^2$-normalized eigenfunction of the full ring of invariant differential operators on $X$ with Laplace eigenvalue $\lambda^2\ge1$, then the local sup-norm bound is
\begin{align}\label{eq: sarnak sup norm}
    \|\psi\|_\infty \lesssim \lambda^{(n-r)/2}.
\end{align}
This upper bound is also sharp in the case when $X$ is of compact type.
The local bound for all $L^p$-norms of such joint eigenfunctions in higher rank was obtained by Marshall \cite{marshall2016p}. Suppose that $X$ is a quotient of the globally symmetric space $S=G/K$, and assume that $S$ is irreducible and not Euclidean. Let $\fa$ be a maximal abelian subspace of the tangent space, and let $\fa^*$ be its dual. Let $B^*\subset \fa^*$ be a compact set that is bounded away from the singular set. Let $\psi\in C^\infty(X)$ be an eigenfunction of the full ring of invariant differential operators on $X$, with $\|\psi\|_{L^2(X)}=1$ and spectral parameter $t\lambda$ where $t$ is large and $\lambda\in B^*$. Marshall proved that, for any $p\geq2$,
\begin{align}\label{eq: Marshall Lp}
    \|\psi\|_{L^p(X)}\lesssim_{B^*,p}
    \begin{cases}
        (\log t)^{1/2} t^{r\delta(n/r,p)},\qquad&p=p_c,\\
        t^{r\delta(n/r,p)},&p\neq p_c,
    \end{cases}
\end{align}
where $p_c=\frac{2(n+r)}{n-r}$ and the function $\delta$ in the exponent is as in \eqref{eq: Sogge bound}. Moreover, these bounds are sharp up to the logarithmic factor in the case when $X$ is of compact type.

Our main result is to remove the factor $(\log t)^{1/2}$ at the critical exponent $p=\frac{2(n+r)}{n-r}$ for $S$ of noncompact type.

\begin{theorem}\label{thm: main 1}
    Suppose that $G$ is a connected simple Lie group of finite center, and $K$ is a maximal compact subgroup, so that $S=G/K$ is an irreducible symmetric space of noncompact type. Let $\Gamma\subset G$ be a co-compact lattice, and let $X=\Gamma\backslash S$. Let $B^*\subset \fa^*$ be a compact set that is bounded away from the singular set. Let $\psi\in C^\infty(X)$ be an eigenfunction of the full ring of invariant differential operators on $X$, with $\|\psi\|_{L^2(X)}=1$ and spectral parameter $t\lambda$ where $t\geq1$ and $\lambda\in B^*$. For $p=p_c=\frac{2(n+r)}{n-r}$ we have
    \[ \|\psi\|_{L^{p_c}(X)}\lesssim_{B^*}
     t^{r\delta(n/r,p_c)} = t^{r/p_c}.\]
\end{theorem}

We note that, by interpolating with the $L^2$-normalization and Sarnak's sup-norm bound \eqref{eq: sarnak sup norm}, Theorem \ref{thm: main 1} implies the $L^p$-norm bounds \eqref{eq: Marshall Lp} without a logarithmic power factor at $p=p_c$.

Similar to the proofs of $\eqref{eq: Sogge bound}$ and \eqref{eq: Marshall Lp}, it suffices to prove a spectral projector estimate.
We take an orthonormal basis $\{\psi_\mu\}_\mu$ for $L^2(X)$, where each $\psi_\mu$ is an $L^2$-normalized  eigenfunction of the full ring of invariant differential operators on $X$ with spectral parameter $\mu\in\fa_\BC^*$.  We still assume that $B^*\subset \fa^*$ is compact and away from the singular set. Given $\lambda\in B^*$ and $t\ge1$ sufficiently large, we define the spectral projector $\bP^Q_{t\lambda}$ (onto the ball of radius $Q\geq1$ centered at $t\lambda$) to be
\begin{align}\label{eq: spectral projector defn}
    \bP^Q_{t\lambda}f = \sum_{\|\mu-t\lambda\|\leq Q}\langle f,\psi_\mu\rangle\psi_\mu,\qquad f\in L^2(X).
\end{align}
We will denote the radius 1 projector by $\bP_{t\lambda}=\bP^Q_{t\lambda}$ ($Q=1$).
Theorem \ref{thm: main 1} follows immediately from the following regular spectral projector estimates.
\begin{theorem}\label{thm: main 2}
    With notation as in Theorem \ref{thm: main 1}, let $\bP_{t\lambda}$ be the spectral projector \eqref{eq: spectral projector defn} with $\lambda\in B^*$ and $t\ge1$. We have
    \begin{align}\label{eq: thm main 2}
         \|\bP_{t\lambda} f\|_{L^{p_c}(X)}\lesssim_{B^*}t^{r/p_c} \|f\|_{L^2(X)}
    \end{align}
    for any $f\in L^2(X)$. By interpolation, we obtain
    \begin{align}\label{eq: thm main 2 general p}
        \|\bP_{t\lambda} f\|_{L^{p}(X)}\lesssim_{B^*,p}t^{r\delta(n/r,p)} \|f\|_{L^2(X)}
    \end{align}
    for any $p\geq2$ and any $f\in L^2(X)$. Moreover, the estimates \eqref{eq: thm main 2 general p} are sharp for any $p\geq2$.
\end{theorem}

\begin{rmk}

    The requirement that $G$ is simple comes from a geometric lemma of Clerc \cite[Proposition 6.1]{Cle87}: the sublevel sets of the Weyl discriminant have finite measure if and only if $G$ is simple. 
    When $G$ is only semisimple, and $S = S_1\times S_2\times\cdots\times S_m$ is a product of irreducible spaces, the $L^p$-norm of an eigenfunction on $X$ as well as the corresponding spectral projector are bounded by the product of the functions \eqref{eq: thm main 2 general p} for each irreducible factor of $S$. Let $n_j$ and $r_j$ be the dimension and the rank of $S_j$.
    Then
    \begin{align*}
        \|\bP_{t\lambda} f\|_{L^{p}(X)}\lesssim_{B^*,p}\prod_{j=1}^mt^{r_j\delta(n_j/r_j,p)} \|f\|_{L^2(X)}
    \end{align*}
    for any $p\geq2$ and any $f\in L^2(X)$.
    We sketch the proof in Section \ref{sec: ss generalization}.
\end{rmk}

Similar to \cite{marshall2016p}, the proof is based on a standard $TT^*$ argument. There is a useful construction used in \cite[Section 4]{anker2023spectral} for the hyperbolic plane. We generalize this idea to a higher-rank setting. Suppose that $k_{t\lambda}\in C_c^\infty(G)$ is a bi-$K$-invariant test function, whose automorphic kernel gives an approximate spectral projector. We consider the Abel transform of $k_{t\lambda}$, which is a function in $C_c^\infty(A)$ that is Weyl invariant. We then dyadically decompose $A$ based on the magnitude of the Weyl discriminant, and multiply each dyadic piece by a complex exponential to construct a holomorphic family of test functions and operators, depending on $z \in \BC$. 

When $z=0$, this recovers the original test function. By Stein's interpolation theorem \cite{Ste56}, we reduce the problem to an $L^1\to L^\infty$ bound when $\Re z =\frac{1}{2}$, and an $L^2\to L^2$ bound when $\Re z =-\frac{r}{n-r}$. Using the Abel transform makes the $L^2\to L^2$ estimate purely a Euclidean problem, avoiding any spherical function analysis. To get rid of the possible $\log t$ factor, we use a partial summation to introduce further cancellation. In the $L^1\to L^\infty$ part, we use oscillatory integral estimates for spherical functions integrated on the spectral side. This motivates us to prove finer estimates for spherical functions in Section \ref{sec:spherical function estimates}, which uses Marshall's local coordinates for $K$ introduced in \cite{marshall2016p}.

Sharpness of \eqref{eq: thm main 2 general p} is proved in Section \ref{sec: sharpness}. The construction is similar to the rank-one situation. Its higher-rank analogue has been mentioned in \cite{marshall2016p}, but only shown for symmetric spaces of compact type. For $p\geq p_c$, we use analogues of zonal functions, which concentrate at a point, to saturate the bound. For $p\leq p_c$, we use analogues of Gaussian beams, which are simply wave packets concentrated on a maximal flat subspace. Note that at $p=p_c$, both types of quasi-modes saturate the bound.

\subsection{Fourier restriction}
It is well known that Sogge's estimates \eqref{eq: Sogge bound} are closely related to $(L^2,L^p)$ Fourier restriction. The method we use to prove Theorem \ref{thm: main 2} also allows us to prove a new Stein--Tomas type result.

Let $G$ and $K$ be as in Theorem \ref{thm: main 1}, with Lie algebras $\fg$ and $\fk$, and the Cartan decomposition $\fg = \fk + \fp$.
We fix a regular element $\Lambda \in \fp$, and let $\cO_\Lambda = \Ad K (\Lambda)$ be the $K$-adjoint orbit of $\Lambda$ in $\fp$, which is diffeomorphic to the boundary of the symmetric space $G/K$. 
We define the Fourier transform for $f\in L^1(\fp)$ by
\[
\widehat{f}(X) = \int_\fp f(Y) e^{-i\langle X,Y\rangle} dY.
\]
Here $dY$ is the Lebesgue measure on the Euclidean space $\fp$ normalized such that the Fourier inversion is given by
\[
f^\vee(X) = \int_\fp f(Y) e^{i\langle X,Y\rangle} dY.
\]
Let $d\mu_{\cO_\Lambda}$ be the measure on $\cO_\Lambda$ induced by the Haar measure on $K$. The $(L^2,L^p)$ Fourier restriction problem asks for which $p$ the bound
\[
    \int_{\cO_\Lambda} |\widehat{f}|^2 d\mu_{\cO_\Lambda} \lesssim \|f\|_{p}^2
    \]
holds for any  $f\in C_c^\infty(\fp)$. By Knapp's homogeneity argument \cite[(1.5)--(1.6)]{DeCarliIosevich1998}, such a result can in general only hold for $1\leq p\leq \frac{2(n+r)}{n+3r}$, and Mockenhaupt \cite{Moc91} proved that this holds whenever $1\leq p< \frac{2(n+r)}{n+3r}$. Our second main result is to prove the endpoint case of this Stein--Tomas restriction theorem by Mockenhaupt.

\begin{theorem}\label{thm: Fourier restriction}
    Let $B\subset\fa$ be a compact set that is bounded away from the singular set.
    Let $\Lambda\in B$ be a regular point and $ p_c' = \frac{2(n+r)}{n+3r}$. Then for any $f\in C_c^\infty(\fp)$ we have
    \[
    \int_{\cO_\Lambda} |\widehat{f}|^2 d\mu_{\cO_\Lambda} \lesssim \|f\|_{p_c'}^2,
    \]
    where the implied constant depends only on $B$.
\end{theorem}

By taking the adjoint, this is equivalent to bounding a Fourier extension operator. For any $g\in C_c^\infty(\fp)$, we let
\[
(gd\mu_{\cO_\Lambda})^\vee (X) =  \int_{\cO_\Lambda} g(\xi) e^{i\langle X,\xi \rangle} d\mu_{\cO_\Lambda}(\xi).
\]

\begin{theorem}\label{thm: Fourier extension}
    Let $B\subset\fa$ be a compact set that is bounded away from the singular set.
    Let $\Lambda\in B$ be a regular point and $ p_c = \frac{2(n+r)}{n-r}$. Then for $g\in C_c^\infty(\fp)$ we have
    \[
    \|(gd\mu_{\cO_\Lambda})^\vee \|^2_{p_c} \lesssim \int_{\cO_\Lambda} |g|^2 d\mu_{\cO_\Lambda},
    \]
    where the implied constant depends only on $B$.
\end{theorem}

To see the relation to the classical sphere Fourier restriction, we take $G$ to be the identity component of $\SO(n,1)$ so that $K$ is isomorphic to $\SO(n)$ and $\fp$ is an $n$-dimensional Euclidean space. Then $\cO_\Lambda\simeq \SO(n)/\SO(n-1)$ is an $(n-1)$-sphere in $\fp$. These classical cases were proved in \cite{Tom75,Tom78} by Tomas and Stein, its general
($L^p,L^p$) formulation has also been studied extensively; see, for example,
\cite{hw} and the references therein.

The sphere Stein--Tomas restriction theorem was generalized to submanifolds of higher codimension by Greenleaf \cite{Gre81}, and to measures by Mockenhaupt \cite{Moc00},  Mitsis \cite{Mit02}, and Bak--Seeger \cite{BK}, based on the uniform decay assumption in every direction of the Fourier transform of the induced measure with a prescribed order. However, the radial estimates are inadequate for higher-rank $\cO_\Lambda$ in Theorem \ref{thm: Fourier restriction}. In fact, the generalized Bessel function estimates by Clerc \cite[Theorem 5.1]{Cle87} show that the decay for the Fourier transform of $d\mu_{\cO_\Lambda}$ weakens near the singular subset of $\fp$.

The classical Stein--Tomas restriction theorem was further strengthened to a Lorentz space estimate by Bak and Seeger \cite{BK}, using Bourgain's interpolation argument introduced in \cite{Bou} in the study of the spherical maximal function. It seems likely that the arguments in the present paper can be combined with
the proof of Proposition 2.1 in \cite{BK} to obtain the corresponding Lorentz
space improvement for our Fourier restriction result. In particular, we expect the $L^{p_c}$ estimates in Theorems 1.1 and 1.2 to remain valid with $L^{p_c}$ replaced by $L^{p_c,2}$.

\subsection{Acknowledgments}
The authors thank Simon Marshall for helpful comments.
P.H. was supported by Deutsche Forschungsgemeinschaft (DFG) Grant No. SFB-TRR 358/1 2023--491392403.
J.H. was supported by the AMS-Simons Travel Grant. X.H. was supported in part by the Simons Foundation and NSF (DMS-2452860).
\subsection{AI Disclosure}
AI was used primarily for advice and editing tasks. The proof and main ideas in the paper were human-generated. The form of Theorem \ref{thm: Weyl Harish-Chandra symbol}, as a modification of \cite[Theorem 1.5]{marshall2016p}, was suggested by AI. 
\section{Notation}

\subsection{Lie groups and algebras}
Let $G$ be a connected simple Lie group of finite center, and let $K$ be a maximal compact subgroup. Let $\fg,\fk$ be their Lie algebras. Let $\langle\,,\,\rangle:\fg\times\fg\to\BR$ be the Killing form. Let $\fp$ be the orthogonal complement of $\fk$ in $\fg$ with respect to the Killing form. The Riemannian structure of $S=G/K$, which is an irreducible globally symmetric space of noncompact type, is given by the Killing form restricted to $\fp$.
Let $\fa$ be a maximal abelian subalgebra in $\fp$. We have the root decomposition
\[
\fg=\fm+\fa+\sum_{\alpha\in\Delta} \fg_\alpha.
\]
Here $\fm\subset\fk$ is the centralizer of $\fa$ in $\fk$, $\Delta\subset\fa^*$ is the restricted root system of $\fg$ with respect to $\fa$ and $\fg_\alpha$ is the $\alpha$-root space. Denote the multiplicity of $\alpha$ by $m_\alpha=\dim \fg_\alpha$. Let $M$ be the centralizer of $\fa$ in $K$, and $M'$ the normalizer of $\fa$ in $K$ so that $W=M'/M$ is the Weyl group.
Let $\Delta^+\subset\Delta$ be a choice of positive roots and let $\fn=\sum_{\alpha\in\Delta^+}\fg_\alpha$. We have the corresponding Iwasawa decomposition of the Lie algebras $\fg=\fn+\fa+\fk$.
Let $N=\exp\fn$ and $A=\exp\fa$, so that $G=NAK$ is the Iwasawa decomposition for $G$. Let $A:G\to\fa$ be the map defined by $g\in N \exp(A(g))K$.
Let $\fa^+\subset \fa$ be the Weyl chamber corresponding to $\Delta^+$, that is
\(
\fa^+=\{X\in\fa\mid \alpha(X)>0\;\forall\alpha\in\Delta^+\},
\)
and let $A^+=\exp\fa^+$.  Let $\rho=\frac{1}{2}\sum_{\alpha\in\Delta^+}m_\alpha \alpha$. We let $n=\dim S$ be the dimension of $S$ and $r=\dim A$ be the rank. It may be seen that $\sum_{\alpha\in\Delta^+}m_\alpha=n-r$. We denote the Weyl discriminant by
\[
\omega (H) = \prod_{\alpha\in\Delta^+} \alpha(H)^{m_\alpha},\quad H\in\fa.
\]
We define $\fa_+^*$ to be the set $\{\nu \in \fa^* \mid \langle\nu,\alpha\rangle>0,\forall \alpha\in\Delta^+  \}$.

Let $dn,da$ be some Haar measures on $N,A$, respectively, and let $dk$ be the probability Haar measure on $K$. We normalize the Haar measure $dg$ on $G$ so that the following integral formula hold. For any $f\in C_c(G)$
\begin{align*}
    \int_G f(g)dg =  \int_{NAK} f(nak)a^{-2\rho} dndadk.
\end{align*}

\subsection{Spherical functions}

The ring of left invariant differential operators on $G/K$, denoted by $\bD(G/K)$,  is isomorphic to $U(\fa)^W$, via the Harish-Chandra homomorphism $\gamma$.
Let $\fa_\BC^*$ be the complexification of $\fa^*$ and we define $\langle\,,\,\rangle: \fa^*_\BC\times\fa\to\BC$ to be the extension of the natural dual pairing. Given $\lambda\in\fa_\BC^*$, there is a unique function $\varphi_\lambda\in C^\infty(G)$ such that $\varphi_\lambda(e)=1$, $\varphi_\lambda$ is bi-$K$-invariant, and
\[
D\varphi_\lambda = \gamma(D)(i\lambda)\varphi_\lambda,\quad\forall D\in \bD(G/K).
\]
Moreover, due to Harish-Chandra $\varphi_\lambda$ can be realized as an average:
\begin{align}\label{eq: HC integral for varphi}
\varphi_\lambda(g)=\int_K e^{\langle i\lambda+\rho,A(kg)\rangle} dk.
\end{align}
Two spherical functions $\varphi_\lambda$ and $\varphi_\mu$ are the same if and only if $\lambda\in W\mu$.

\subsection{Harish-Chandra transform}
The Harish-Chandra transform is defined by
\[
\sH f(s)=\int_G f(g)\varphi_{-s}(g) dg.
\]
for functions $f$ on $G$. By the Paley--Wiener theorem of Gangolli \cite{Gan71},  $\sH$ is an isomorphism from the space of bi-$K$-invariant smooth compactly supported functions
$C_c^\infty(G\sslash K)$ onto the  Weyl invariant  Paley--Wiener space  $ \cP\cW(\fa^*)^W$. We have an integral formula for its inversion $\sH^{-1}$:
\[
f(g)=\frac{1}{|W|}\int_{\fa^*} \sH f(s) \varphi_s(g) d\mu(s)=\int_{\fa_+^*} \sH f(s) \varphi_s(g) d\mu(s),
\]
where $d\mu(s)=c|\bc(s)|^{-2}ds$ the is Plancherel measure. Here $c$ is some constant and $\bc$ is the Harish-Chandra $\bc$-function.

\subsection{Maass forms}
Suppose $\Gamma\subset G$ is a cocompact lattice, and let $X=\Gamma\backslash S$. We say a function $\psi_\lambda\in C^\infty(X)$ is a Maass form of spectral parameter $\lambda\in\fa^*_\BC$ if $\psi_\lambda$ is a joint eigenfunction of the ring $\bD(G/K)$ and satisfies
\[
D\psi_\lambda = \gamma(D)(i\lambda) \psi_\lambda,\qquad\forall D\in\bD(G/K).
\]
By the spectral theory, $L^2(X)$ has an orthonormal basis of Maass forms $\{\psi_\mu\}_\mu$ for $\mu$ (with multiplicities) in the unitary spectrum.

\section{Spherical function estimates}\label{sec:spherical function estimates}

Let $B\subset\fa$ and $B^*\subset \fa^*_+$ be compact, and let $c \in C^\infty(B \times B^* \times K)$ and $c'\in C^\infty( B^* \times K)$. Set
\[
 I_t(c;X,\nu)=\int_K e^{it\langle\nu,A(k\exp X)\rangle}
                  c(X,\nu,k)dk,
\]
and
\[
 E(c';X,\nu)=\int_K e^{i\langle\nu,\Ad k X\rangle}
                  c'(\nu,k)dk.
\]

In this section, we prove the following oscillatory integral estimates, which can be applied to spherical functions.
\begin{theorem}\label{thm: Weyl Harish-Chandra symbol}
For $t\geq1$ let $c$ and $I_t(c;X,\nu)$ be as above, and suppose that $X\in B$ and $\nu\in B^*$.
There exist coefficients $a_{w,t}(X,\nu)$, smooth in $\nu$ on a neighborhood of $B^*$, such that
\begin{equation}\label{eq:target}
 I_t(c;X,\nu)=\sum_{w\in W}e^{it\langle\nu,wX\rangle}
                      a_{w,t}(X,\nu),
\end{equation}
and for every multi-index $\gamma$, $X\in B$, and $\nu\in B^*$,
\begin{equation}\label{eq:symbol}
    |\partial_\nu^\gamma a_{w,t}(X,\nu)|\lesssim \prod_{\alpha\in\Delta^+}(1+t|\alpha(X)|)^{-m_\alpha/2},
\end{equation}
where the implied constant depends only on $\gamma$, $B$, $B^*$, and $c$.
\end{theorem}

By a theorem of Duistermaat \cite{Dui84}, this is easily seen to be equivalent to a similar estimate for $E(c';X,\nu)$. 


\begin{theorem}[\cite{Dui84}]
    Let $\pi:\fp\to\fa$ be the orthogonal projection.
    There is a real analytic map $\Psi:\fp\to K$ such that the following holds.
    \begin{enumerate}
        \item For each $X\in\fp$, $k\mapsto \Psi(\Ad k(X)) k$ is a diffeomorphism from $K$ onto $K$.
        \item  For all $X\in\fp$, $A\circ\exp(\Ad \Psi (X)(X)) = \pi(X)$.
    \end{enumerate}
\end{theorem}

\begin{theorem}\label{thm: Weyl Harish-Chandra symbol 2}
Let $c'$ and $E(c';X,\nu)$ be as above, and suppose that $X\in \fa$ and $\nu\in B^*$.
There exist coefficients $b_{w}(X,\nu)$, smooth in $\nu$ on a neighborhood of $B^*$, such that
\begin{equation}\label{eq:target 2}
 E(c';X,\nu)=\sum_{w\in W}e^{i\langle\nu,wX\rangle}
                      b_{w}(X,\nu),
\end{equation}
and for every multi-index $\gamma$, $X\in \fa$, and $\nu\in B^*$,
\begin{equation}\label{eq:symbol 2}
    |\partial_\nu^\gamma b_{w}(X,\nu)|\lesssim \prod_{\alpha\in\Delta^+}(1+|\alpha(X)|)^{-m_\alpha/2},
\end{equation}
where the implied constant depends only on $\gamma$, $B^*$, and $c'$.
\end{theorem}
\begin{proof}
    If $\|X\|\geq1$, we let $t=\|X\|$ and $H= X/\|X\|$. Let $\Psi:\fp\to K$ be the map in Duistermaat's theorem. Therefore, by letting $\ell=\Psi(\Ad k(H))k$, we have
    \begin{align*}
        E(c';X,\nu)=\int_K e^{it\langle\nu,\Ad k H\rangle}
                  c'(\nu,k)dk=\int_K e^{it\langle\nu,A(\ell e^H)\rangle}
                  c(H,\nu,\ell)d\ell
    \end{align*}
    for some new amplitude $c$. Then the theorem follows from Theorem \ref{thm: Weyl Harish-Chandra symbol} immediately.

    If $\|X\|\leq1$, then the same proof as above without rescaling can be applied.
\end{proof}

The above results are similar to \cite[Theorem 1.5]{marshall2016p}, but for our purposes we need a bound for the $\nu$-derivatives rather than the $X$-derivatives. Before proving Theorem \ref{thm: Weyl Harish-Chandra symbol} let us review the geometry involved. Let \[
\phi (X, \nu, \ell) := \langle \nu, A(\ell \exp X) \rangle,\qquad X\in\fa,\nu\in\fa^*,\ell\in K,
\]
be the phase in $I_t$. When we speak of the critical set we will always mean $D_\ell \phi = 0$. For generic values of $X$ and $\nu$ the critical set is precisely $M'$, with the critical values $\langle \nu, wX \rangle$ being constant on the cosets $wM$. As $X$ and $\nu$ become singular, the values $\langle \nu, wX \rangle$ coincide, the gradient in the regions between the connected components vanishes, and they merge into a smaller number of higher-dimensional components. For a fixed $X$ and $\nu$ the situation is summarized by the following proposition:

\begin{proposition}[\cite{DKV83} Proposition 5.4]\label{prop:dkvCosets}
    The critical set of $\phi(X,\nu,\ell)$ is a disjoint union 
    \[
    \coprod_{w \in W_\nu \backslash W / W_X} Z_K(\nu) w Z_K(X)
    \]
    where $W = M'/M$ and $W_H = \{ w \in W : wH = H \}$.
\end{proposition}

In order to prove Theorem \ref{thm: Weyl Harish-Chandra symbol} we require a more quantitative version of the above fact. This is supplied by  \cite[Theorem 4.7]{marshall2016p}, which gives a local change of coordinates making $\phi$ into a quadratic polynomial in $\ell$, whose coefficients depend in an explicit way on $X$. We restate the theorem here with some notational simplifications; in particular, we replace Marshall's complexified spaces with real subspaces (see \cite[Remark 4.8]{marshall2016p}) and subsume \cite[Proposition 4.22]{marshall2016p} into a special case of \cite[Theorem 4.7]{marshall2016p}.

\begin{proof}[Proof of Theorem \ref{thm: Weyl Harish-Chandra symbol}]

Fix an ordering of the simple roots, $\beta_1,\dots,\beta_r$. Define a cone
\[
\cC = \{ X \in \fa : \beta_1(X) \geq \dots \geq \beta_r(X) \geq 0 \}.
\]
By partitioning $B$ and using Weyl invariance, we may assume without loss of generality that $B \subset \cC$.
(This partitioning poses no problems because we do not actually need $c$ to be smooth in $X$, rather, we only need derivatives in $\nu$ and $\ell$ to be bounded with respect to $X$).   
    
Fix basepoints $\ell^{\circ} \in K, X^\circ \in \cC, \nu^\circ \in B^*$. Define $\fa_{\ell^{\circ}} = \{ H \in \fa : \Ad_{\ell^{\circ}} H \in \fa \}$.  Then $\ell^{\circ} = w \ell_0$ for some $\ell_0 \in Z_K(\fa_{\ell^\circ})$ and $w \in M'$. Note that $D_{\ell} \phi( X, \nu, \ell^{\circ}) = 0$ if and only if $X \in \fa_{\ell^{\circ}}$. These facts can be found in \cite[Lemmas 4.1 and 4.4]{marshall2016p}.

    Let $\hat{\beta}_1,\dots,\hat{\beta}_r \in \fa$ be the simple coweights, i.e. $\beta_i(\hat{\beta}_j) = \delta_{ij}$. Define a flag $\{0\} = V_0 \subset \dots \subset V_r$ by $V_i = \langle \hat{\beta_1}, \dots, \hat{\beta}_i \rangle$. Introduce a new space $\cA \cong \mathbb{R}^r$, which will serve as a subset of the blowup of $\fa$ along the $V_i$. Given a point $u = ({u}_{0}, \dots, {u}_{r-1}) \in \cA$ we define the blowdown map $\pi : \cA \to \fa$ by 
    \[
    \pi(u) = \sum_j  u_0 \dots u_{j-1} \hat{\beta}_j.
    \]
    We also define a right inverse $\iota : \cC \to [0,\infty) \times [0,1]^{r-1}$ by 
    \begin{align}\label{eq:lift}
    \iota(X)_i &= 
    \begin{cases}
    \beta_1(X) & i = 0
           \\
    \frac{\beta_{i+1}(X)}{\beta_i(X)} & 1 \leq i \leq r-1
    \end{cases}
    \end{align}
    with the convention that $0/0=1$ and set $u^\circ = \iota(X^\circ)$. (Note that we do not even need $\iota$ to be continuous).

    Let $0 \leq p \leq r$ be the largest index such that $V_p \subset \fa_{\ell^{\circ}}$ and define $Q(u) = {u}_{0} \dots {u}_{p}$ if $p < r$, or $Q = 0$ if $p=r$. Define $q$ by 
    \[
    q := 
    \begin{cases}
    r & p = r
        \\ 
    \max \{ i \leq p : u^\circ_i = 0 \} & p < r, Q(u^\circ) = 0
        \\
    0 & p < r, Q(u^\circ) \neq 0
    \end{cases}
    .
    \]
    Let $\alpha_1, \dots, \alpha_d$ be an enumeration of the positive roots, including multiplicity, that are non-vanishing on the subspace $w V_q$.

    We should explain what the 3 cases for $q$ mean: in that order, they correspond to whether $\ell^\circ$ is always, sometimes, or never a critical point, as we vary $X$. The list of roots is correspondingly all roots, some roots, or empty. 

    \begin{theorem}[\cite{marshall2016p} Theorem 4.7 and Proposition 4.22]\label{thm:simon47}
        For all $u^\circ \in [0,\infty)\times [0,1]^{r-1}, \nu^{\circ} \in B^*, \ell^{\circ} \in K$, there exists an analytic map $f_{u,\nu}(\ell) = f(u,\nu,\ell)$ to $\BR^{\dim K}$, defined in an open neighborhood of $(u^{\circ}, \nu^{\circ}, \ell^{\circ})\in\cA \times \fa^* \times K$ with the following properties. We have $f_{u^{\circ},\nu^{\circ}}(\ell^{\circ}) = 0$ and $f_{u,\nu}$ is a local embedding for all $(u,\nu)$. Write $z = f_{u,\nu}(\ell)$. Then
        \[
        \phi(\pi(u), \nu, \ell) 
        = 
        \phi_S(u, \nu) - 
        \sum_{i=1}^{d}
        \langle \nu, \alpha_i \rangle 
        \alpha_i(w\pi(u)) 
        {z}_{i}^2 
        + Q(u) L({z}_{d+1},\dots,{z}_{\dim K}) 
        .
        \]
        where $\phi_S$ is some analytic function, and $L$ is a non-constant affine linear function (unless $Q=0$). 
    \end{theorem}
    We will need a slight strengthening: in the above theorem, we can replace $\phi_S$ with the phase $\langle \nu, w \pi(u) \rangle$, where $w$ is the Weyl element associated to $\ell^{\circ}$ above. 
    \begin{proposition}
        There is an analytic function $R(u,\nu)$ defined near $(u^{\circ},\nu^{\circ})$ such that
        \[
        R(u,\nu)Q(u) = \phi_S(u, \nu) - \langle \nu, w \pi(u) \rangle
        .
        \]
    \end{proposition}
    \begin{proof}
    If $p = r$, then $Q=0$. This case is covered by Proposition 4.22 in \cite{marshall2016p}. 
    
    Assume $p < r$ and fix $i \leq p$. We will show that the right hand side is divisible by ${u}_{i}$. If $u^\circ_i \neq 0$, then $u_i$ is bounded away from $0$ in the neighborhood and there is nothing to prove. 
    
    The proof of Theorem 4.7 in \cite{marshall2016p}, immediately after Equation (4.13), defines a map $\xi : \cA \times \fa^* \to \cA \times \fa^* \times K$ near $(u^{\circ},\nu^{\circ},\ell^{\circ})$. Let $\ell_*(u,\nu)\in K$ be the last coordinate of $\xi(u,\nu)$. The last sentence of the proof of Theorem 4.7 defines $\phi_S(u,\nu) := \phi(\pi(u),\nu,\ell_*(u,\nu))$.
    \cite[Corollary 4.15(d)]{marshall2016p} implies\footnote{In Marshall's notation $Y_{s'} \subseteq \ell K_{i,\BC}$ is implicitly projecting $Y_{s'}$ to its $K$-coordinate, and $\xi(s')$ lies in $Y_{s'}$. The notation $s \in S_i$ is equivalent to ${u}_{i} = 0$. } that when ${u}_{i} = 0$, we have $\ell_*(u,\nu) \in \ell^{\circ} Z_K(V_i)$. Thus we may write $\ell_*(u,\nu) = \ell^{\circ} \ell_1 = w \ell_0 \ell_1$ where $\ell_1 \in Z_K(V_i)$. Note also that ${u}_{i} = 0$ implies $X:=\pi(u) \in V_i$. We write
    \begin{align*}
        A(\ell_*(u,\nu) \exp X) 
        &=
        A(w \ell_0 \ell_1 \exp X)
        \\
        &=
        A(w (\exp X) \ell_0 \ell_1)
        \\
        &=
        A((\exp w X) w\ell_0 \ell_1)
        \\
        &=
        wX
        ,
    \end{align*}
    where in the second line we use that $\ell_1 \in Z_K(V_i)$, $\ell_0 \in Z_K(\fa_{\ell^{\circ}})$ and $V_i \subset \fa_{\ell^{\circ}}$. Thus $ \phi_S(u, \nu) - \langle \nu, wX \rangle = \phi(X, \nu, \ell_*(u,\nu)) - \langle \nu, wX \rangle = 0$. 
    \end{proof}

    Now having set up Theorem \ref{thm:simon47}, Theorem \ref{thm: Weyl Harish-Chandra symbol} follows from a stationary phase calculation. Let $\cB \subset \cA$ be the closure of $ \iota(B)$ and let $\cS = \cB\times B^* \times K$. Using compactness, take a finite collection of basepoints $(u^\circ_j, \nu^\circ_j, \ell^\circ_j) \in \cS$ whose associated neighborhoods cover $\cS$. Let $\sum_j \chi_j(u,\nu,\ell)$ be a partition of unity adapted to the cover, satisfying $\sum_j \chi_j(u,\nu,\ell) = 1$ on $\cS$ and set 
    \[
    c_j(u,\nu,z) := \chi_j(u,\nu,f_{u,\nu}^{-1}(z)) c(\pi(u),\nu,f_{u,\nu}^{-1}(z))
    .
    \]
    Then 
    \[
    I_t(c;X,\nu) = 
    \sum_j 
     \int_{\BR^{\dim K}} 
    \exp 
    \left(
    it\phi(\pi(u),\nu,f_{u,\nu}^{-1}(z))
    \right)
    c_j(u,\nu,z)
    J(u,\nu,z)
    dz
    \]
    where $J$ is some smooth Jacobian. Applying Theorem \ref{thm:simon47} to each summand, we obtain
    \[
    I_t(c;X,\nu) = \sum_j e^{it\langle \nu, w_j \pi(u) \rangle}
    \tilde{a}_{j,t}(u,\nu)
    \]
    where $w_j$ is the Weyl element associated to $\ell^\circ_j$, and
    \begin{align*}
    \tilde{a}_{j,t}(u,\nu) :=
    \int_{\BR^{\dim K}}
    \exp \left(
    it
    \left(
        -\sum_{i=1}^d
        \langle \nu, \alpha_i \rangle 
        \alpha_i (w_j \pi(u)) 
        {z}_{i}^2 
        + Q(u) [L_j({z}_{d+1},\dots,{z}_{\dim K}) + R(u,\nu)]
    \right)
    \right)
    \\
    \times c_j(u,\nu,z)
    J(u,\nu,z) 
    dz
    .
    \end{align*}
    Here $d,\alpha_i,Q,L,R$ also depend on $j$ but we suppress the subscript. 
    
    We are finally ready to perform integration by parts. Note that $c_j, R$ and $J$ are smooth in $\nu$. Since $\nu \in B^*$ avoids the singular set, $\langle \nu, \alpha_i \rangle$ is bounded away from 0. Thus, applying stationary phase in $z_i$ gives a decay of $(1+t\abs{\alpha_i(w \pi(u))})^{-1/2}$ for each $1 \leq i \leq d$. Meanwhile, if $Q \neq 0$, then $\partial_{z_k} L \neq 0$ for some $k \geq d+1$, and applying nonstationary phase in $z_k$ gives rapid decay in $t \abs{Q(u)}$. Thus
    \begin{align}\label{eq: Q(u) prod bound}
        \abs{
    \tilde{a}_{j,t}(u,\nu) 
    }
    \lesssim_N
    \bigl(1+t|Q(u)|\bigr)^{-N}
    \prod_{i=1}^{d}
    \bigl(1+t|\alpha_i(w\pi(u))|\bigr)^{-1/2}.
    \end{align}
The same bound applies after taking arbitrary derivatives in $\nu$: substitute $y_i = \langle \nu, \alpha_i \rangle^{1/2} z_i$. Since $\nu \in B^*$ the Jacobian is bounded. Differentiating in $\nu$, we get factors $D_\nu R(u,\nu)$ and $t\abs{Q(u)}$ in the amplitude. The former is bounded, while the latter may be canceled by taking $N$ larger. 

    It remains to show that the right-hand side of \eqref{eq: Q(u) prod bound} is bounded by 
    \[
    \lesssim 
    \prod_{\alpha \in \Delta^+} (1+t \abs{\alpha(X)})^{-m_\alpha/2}
    .
    \]
    The roots missing from the product on the left-hand side are those $\alpha$ with $\alpha(w V_q) = 0$. (If $Q=0$, then $d=n-r$ and all roots are already included). We have
    \begin{align*}
        \alpha(w \pi(u)) 
        &= \sum_j \alpha(w\hat{\beta}_j) u_0 \dots u_{j-1}
        \\
        &= u_0 \dots u_q \sum_j \alpha(w\hat{\beta}_j) u_{q+1} \dots u_{j-1},
    \end{align*}
    and $u_{q+1},\dots,u_{j-1}$ are bounded on the neighborhood, hence $\abs{\alpha(w \pi(u))} \lesssim \abs{u_0 \dots u_q}$ on the neighborhood. By the definition of $q$, we have $u_i^\circ \neq 0$ for $q+1 \leq i \leq p$, and by demanding sufficiently small neighborhoods we can assume $\abs{u_i} > \abs{u_i^\circ}/2$ for $q+1 \leq i \leq p$. Hence $\abs{u_{q+1}\dots u_p} \gtrsim 1$ and $\abs{\alpha(w \pi(u))} \lesssim \abs{u_0 \dots u_q} \lesssim \abs{u_0 \dots u_p} = \abs{Q(u)}$, with the implicit constant depending only on the root and the neighborhood. Thus, by taking $N$ large enough, we may bound the right-hand side of \eqref{eq: Q(u) prod bound} by
    \begin{align*}
    &\lesssim 
    \prod_{\alpha \in \Delta^+} (1+t \abs{\alpha(w\pi(u))})^{-m_\alpha/2}
     =\prod_{\alpha \in \Delta^+} (1+t \abs{\alpha(X)})^{-m_\alpha/2}
    \end{align*}
    as desired. Finally, setting
    \[
    a_{w,t}(X,\nu) := \sum_{j : w_j = w} \tilde{a}_{j,t}(\iota(X), \nu)
    \]
    gives the theorem. 
\end{proof}

\section{Symbol estimates}\label{sec 4}

In this section, we prove some symbol estimates, which may be applied in the same way to bound $L^2\to L^2$ norms in the next two sections.

We first recall a finite volume result of Jean-Louis Clerc, and the proof from \cite[\S6]{Cle87}, which is the source of our requirement that $\fg$ is simple.

\begin{proposition}[\cite{Cle87} Proposition 6.1]\label{prop: cle volume estimate}
    If $\fg$ is simple, the set $\{H\in\fa\mid |\omega(H)|\leq1  \}$ has finite Lebesgue measure.
\end{proposition}
It is also easy to see that the converse holds. In the simplest case of $\mathfrak{sl}(2,\BR) \times \fsl(2,\BR)$, the set is $\{ (x,y) \mid \abs{xy} \leq 1 \}$, and a similar situation occurs whenever $\fg$ has a nontrivial factorization. Note that $\omega$ is homogeneous of degree $n-r$, so by rescaling we have for arbitrary $c>0$
\begin{align}\label{eq: volume estimates}
    \Vol( \{ H\in\fa\mid |\omega(H)|\leq c \} )\lesssim c^{\frac{r}{n-r}}.
\end{align}

We will prove a refinement of Proposition \ref{prop: cle volume estimate}, which is of the form $\Vol(\{|\omega(H) | \leq 1, \norm{H} \geq R \})\lesssim R^{-\delta}$ for any $R\geq1$ and a fixed $\delta>0$. We prove this from the same strategy of \cite{Cle87} (see also \cite{Rot74,Rot75} when $\fg$ is a complex simple Lie algebra). Let $S=\{\alpha_1,\dots,\alpha_r \}\subset\Delta^+$ be the set of simple roots.  If $\alpha\in\Delta^+$ with $\alpha=\sum_{i=1}^r n_i\alpha_i$, the support of $\alpha$ is defined by $S(\alpha)=\{i\mid n_i\neq0 \}$. If $\Theta$ is a subset of $\{1,2,\dots,r\}$, let $N(\Theta)$ be the number of positive roots, counted with multiplicity, whose supports are contained in $\Theta$.

\begin{proposition}[\cite{Cle87} Proposition 6.2]\label{prop: Cle lie algebra lemma}
    If $\Theta$ is a nonempty proper subset of $\{ 1,2,\cdots,r\}$, we have
    \[
    \frac{N(\Theta)}{|\Theta|}<\frac{N(\{1,2,\dots,r\})}{r}
    \]
    where
    \[
        N(\{1,2,\dots,r\})=\sum_{\alpha\in\Delta^+}m_\alpha=n-r.
    \]
\end{proposition}

Thanks to Proposition \ref{prop: Cle lie algebra lemma}, we can find a constant $\epsilon(\fg)>0$, which only depends on $\fg$, such that the following inequalities hold. For any nonempty and proper subset $\Theta\subset \{ 1,2,\cdots,r\}$, we have
\begin{align}
    \frac{rN(\Theta)}{n-r}-|\Theta|+\epsilon(\fg)<0.\label{eq: defn of epsilon g}
\end{align}

\begin{lemma}\label{lem: tail volume}
    Let $\fg$ be a simple Lie algebra of real rank $r\geq2$.
    Let $\epsilon(\fg)>0$ be a constant satisfying \eqref{eq: defn of epsilon g}.
    We have for any $R\geq1$
    \begin{align*}
        \Vol(\{ H\in\fa\mid |\omega(H)|\leq 1,\; \|H\|\geq R \} ) \lesssim R^{-\frac{\epsilon(\fg)}{(r-1)}}.
    \end{align*}
\end{lemma}

\begin{remark}
Indeed, by Proposition \ref{prop: cle volume estimate}, the measure must decay with $R$, so the only question is whether the decay is a power. This actually holds for any set defined by polynomial inequalities, by \cite[Proposition 1.5]{ CluckersMiller2011}. The measure of the set is easily seen to be a \textit{constructible function} of $R$, and applying their proposition with $X = \{ 0 \}, y = R$ gives $\lesssim R^{-\epsilon}$ for some ineffective $\epsilon > 0$. 
\end{remark}

\begin{proof}

    We may restrict to the positive Weyl chamber $\fa^+$. We can identify $\fa^+$ with $\BR^r_{>0}$ via the simple root coordinates  
    \begin{align}\label{eq: simple root coordinates}
        H\mapsto x=(x_1,\cdots,x_r),\qquad\text{ where }x_i=\alpha_i(H).
    \end{align}
    Under this realization, $\omega(H)$ is a product of linear forms with nonnegative coefficients, and we denote the corresponding polynomial in $\BR_{>0}^r$ by $\omega(x)$. Therefore, the volume considered restricted to $\fa^+$ is bounded by
    \begin{align*}
        &\lesssim
        \Vol(\{x\in\BR_{>0}^r\mid \omega(x)\leq1,\;x_i>0,\;\|x\|\geq R\})
        \\
        &\lesssim
        \sum_{\pi \in S_r} 
        \Vol(\{\omega(x)\leq1,\;0<x_{\pi(1)}\leq x_{\pi(2)}\leq\cdots\leq x_{\pi(r)},\;\|x\|\geq R\}).
    \end{align*}
    We consider the region $0<x_{1}\leq x_{2}\leq\cdots\leq x_{r}$ as the other estimates are the same.  In this region, we have $x_r^2< \|x\|^2\leq r x_r^2$, so
    \begin{align}\label{eq: region 1 volume}
        &\Vol(\{\omega(x)\leq1,\;0<x_{1}\leq x_{2}\leq\cdots\leq x_{r},\;\|x\|\geq R\})\notag\\
        \leq&\Vol(\{\omega(x)\leq1,\;0<x_{1}\leq x_{2}\leq\cdots\leq x_{r},\;x_r\geq R/\sqrt{r}\}).
    \end{align}
    By a change of variables
    \begin{align}\label{eq: change x to y}
    y_1=x_1,\; y_2=\frac{x_2}{x_1},\;\cdots,\;y_r=\frac{x_r}{x_{r-1}},
    \end{align}
    whose Jacobian is equal to $y_1^{r-1}y_2^{r-2}\cdots y_{r-1}$, the region in \eqref{eq: region 1 volume} becomes
    \begin{align*}
        \{\omega(x)\leq1,\; y_1>0,\; y_2,\cdots,y_r\geq1,\; y_1y_2\cdots y_r\geq R/\sqrt{r} \}
    \end{align*}
    To deal with the condition $\omega(x)\leq1$, we write $N_i = N(\{1,2,\cdots,i\})$, and notice that
    \[
    \omega(x)\geq x_1^{N_1} x_2^{N_2-N_1}\cdots x_r^{N_r-N_{r-1}} =
    \left(y_1y_2^{1-\frac{N_1}{N_r}}\cdots y_r^{1-\frac{N_{r-1}}{N_r}}\right)^{N_r}.
    \]
    Therefore, the volume \eqref{eq: region 1 volume} is bounded by the integral of the Jacobian $y_1^{r-1}y_2^{r-2}\cdots y_{r-1}$ over the region
    \begin{align}\label{eq: region 2} 
       \left\{0<y_1\leq y_2^{\frac{N_1}{N_r}-1}\cdots y_r^{\frac{N_{r-1}}{N_r}-1},\; y_2,\cdots,y_r\geq1,\; y_1y_2\cdots y_r\geq \frac{R}{\sqrt{r}} \right\}.
    \end{align}
    Combining the first and the third inequalities above and using $\frac{N_{k-1}}{N_r}<1$, we obtain
    \begin{align*}
        y_2\cdots y_r\geq \frac{R}{\sqrt{r}},
    \end{align*}
    so the region \eqref{eq: region 2} is contained in the larger region
    \begin{align*}
        \bigcup_{k=2}^r\left\{0<y_1\leq y_2^{\frac{N_1}{N_r}-1}\cdots y_r^{\frac{N_{r-1}}{N_r}-1},\; y_2,\cdots,y_r\geq1,\; y_k\geq ({R}/\sqrt{r})^{1/(r-1)} \right\}
    \end{align*}
    We thus summarize that the volume \eqref{eq: region 1 volume} is bounded by
    \begin{align*}
        \lesssim&\sum_{k=2}^r\int_1^\infty\cdots\int_{R^{1/(r-1)}}^\infty\cdots\int_1^\infty\left(\int_{0}^{y_2^{\frac{N_1}{N_r}-1}\cdots y_r^{\frac{N_{r-1}}{N_r}-1}}  y_1^{r-1}dy_1\right) y_2^{r-2}\cdots y_{r-1}dy_2\cdots dy_r\\
        =&r^{-1} \sum_{k=2}^r\left(\int_1^\infty y_2^{\frac{rN_{1}}{N_r}-2} dy_2\right)\cdots \left(\int_{R^{1/(r-1)}}^\infty y_k^{\frac{rN_{k-1}}{N_r}-k} dy_k\right)\cdots\left(\int_{1}^\infty y_r^{\frac{rN_{r-1}}{N_r}-r} dy_r\right).
    \end{align*} 
    Using the positive constant $\epsilon(\fg)$ and the inequality \eqref{eq: defn of epsilon g}, for any $2\leq k\leq r$, we have
    \[
    \frac{rN_{k-1}}{N_r}-k<-1-\epsilon(\fg),
    \]
    so the volume considered in this lemma is $$\lesssim\int_{R^{1/(r-1)}}^\infty y^{-1-\epsilon(\fg)} dy\lesssim R^{\frac{-\epsilon(\fg)}{(r-1)}},$$ which completes the proof.
\end{proof}

\begin{lemma}\label{lem: finite log integral}
    The integral
    \[
    \int_{|\omega(H)|\leq1}\log(1+\|H\|) dH
    \]
    is convergent.
\end{lemma}
\begin{proof}
    For rank one, take $\delta(\fg)=1$. By Lemma \ref{lem: tail volume} for $r\geq2$, trivially for rank one, we have
    \begin{align*}
        \Vol( \{ H\in\fa\mid |\omega(H)|\leq 1,\; \|H\|\geq R \} )\lesssim R^{-\delta(\fg)}.
    \end{align*}
    Therefore,
    \begin{align*}
        \int_{|\omega(H)|\leq1}\log(1+\|H\|) dH \lesssim 1+ \sum_{k=0}^\infty \int_{2^k\leq \| H\|\leq 2^{k+1},\;|\omega(H)|\leq1} \log(1+\|H\|) dH\\
        \lesssim 1+ \sum_{k=0}^\infty \log(1+2^{k+1})(2^k)^{-\delta(\fg)}\lesssim 1,
    \end{align*}
    which is convergent.
\end{proof}

We use the tail volume estimate in Lemma \ref{lem: tail volume} to establish a Fourier decay estimate, which is also trivial in rank one.

\begin{lemma}\label{lem: bd for hm via IBP}
    Let $b_0\in C_c^\infty(\BR)$ be a nonnegative even cutoff function that is equal to 1 on $[-1,1]$ and zero outside $[-2,2]$, and let $b_1(x) = b_0(x)-b_0(2x)$. 
      Let $\chi\in C^\infty(\fa)$ and $c>0$. There exists $\delta > 0$ such that for any $\xi\geq c$ and $\nu\in\fa^*$ we have
    \begin{align*}
        \int_\fa \chi(H)b_1(\xi\omega(H)) e^{i\langle\nu,H\rangle} dH\lesssim \xi^{-\frac{r}{n-r}}\left(1+\|\xi^{-\frac{1}{n-r}} \nu\|\right)^{-\delta}.
    \end{align*}
    Here the implied constant depends only on $c$, the sup-norms of $\chi$ and $b_0$ and their derivatives. 
\end{lemma}
\begin{proof}
    By changing variables $H\mapsto \xi^{-\frac{1}{n-r}}H$, we have
    \begin{align*}
        \int_\fa \chi(H)b_1(\xi\omega(H)) e^{i\langle\nu,H\rangle} dH = \xi^{-\frac{r}{n-r}}\int_\fa \chi(\xi^{-\frac{1}{n-r}}H)b_1(\omega(H))e^{i\langle\xi^{-\frac{1}{n-r}}\nu,H\rangle} dH,
    \end{align*}
    so it suffices to show that, for $\nu\in\fa^*$ with $\|\nu\|\geq1$,
    \begin{align}\label{eq: ibp bound after changing variables}
        \int_\fa \chi(\eta^{-1}H)b_1(\omega(H))e^{i\langle\nu,H\rangle} dH\lesssim \|\nu\|^{-\delta}
    \end{align}
    holds uniformly for $\eta\gtrsim1$.
    
    Let $R\geq1$, which will be chosen explicitly later. By Lemma \ref{lem: tail volume} for $r\geq2$, trivially for $r=1$, we have 
    \begin{align}\label{eq: ibp bound after changing variables 1}
        \int_\fa \big(1-b_0(R^{-1}\|H\|)\big)\chi(\eta^{-1}H)b_1(\omega(H))e^{i\langle\nu,H\rangle} dH\lesssim R^{-\delta_1}
        ,
    \end{align}
    for some $\delta_1 > 0$, which may be taken to be $1$ in the rank 1 case. 
    We bound the remaining part by integrating by parts. By writing $\nabla_H$ and $\rm div_H$ for taking the gradient and the divergence with respect to $H$ on $\fa$, we notice that
    \[
    \mathrm{div}_H\left( \nabla_H e^{i\langle\nu,H\rangle}\right) = -\|\nu\|^2 e^{i\langle\nu,H\rangle}.
    \]
    By the divergence theorem, we have
    \begin{align*}
        &\int_\fa b_0(R^{-1}\|H\|)\chi(\eta^{-1}H)b_1(\omega(H))e^{i\langle\nu,H\rangle} dH\\
        =&-\frac{1}{\|\nu\|^2}\int_\fa b_0(R^{-1}\|H\|) \chi(\eta^{-1}H) b_1(\omega(H)) \mathrm{div}_H\left( \nabla_H e^{i\langle\nu,H\rangle}\right) dH\\
        =&\frac{1}{\|\nu\|^2}\int_\fa \langle\nabla_H\left( b_0(R^{-1}\|H\|) \chi(\eta^{-1}H) b_1(\omega(H)) \right), \nabla_H e^{i\langle\nu,H\rangle}\rangle dH.
    \end{align*}
    It is clear that $\|\nabla_H e^{i\langle\nu,H\rangle}\|\lesssim \|\nu\|$ and
    \begin{multline*}
    \|\nabla_H\left( b_0(R^{-1}\|H\|) \chi(\eta^{-1}H) b_1(\omega(H)) \right)\|\\
    \lesssim 1+ \|b_0(R^{-1}\|H\|) \chi(\eta^{-1}H) \nabla_H\left(  b_1(\omega(H)) \right)\|    \lesssim R^{n-r-1},
    \end{multline*}
    so we get
    \begin{align}\label{eq: ibp bound after changing variables 2}
        \int_\fa b_0(R^{-1}\|H\|)\chi(\eta^{-1}H)b_1(\omega(H))e^{i\langle\nu,H\rangle} dH\lesssim R^{n-r-1}\|\nu\|^{-1}.
    \end{align}
    Combining \eqref{eq: ibp bound after changing variables 1} and \eqref{eq: ibp bound after changing variables 2} gives
     \begin{align*}
        \int_\fa \chi(\eta^{-1}H)b_1(\omega(H))e^{i\langle\nu,H\rangle} dH\lesssim R^{-\delta_1}+R^{n-r-1}\|\nu\|^{-1},
    \end{align*}
    so we complete the proof by choosing $R=\|\nu\|^{1/(\delta_1+n-r-1)}$.
\end{proof}

\begin{proposition}\label{prop: symbol estimates}
    Let $b_0\in C_c^\infty(\BR)$ be a nonnegative even cutoff function that is equal to 1 on $[-1,1]$ and zero outside $[-2,2]$, and let $b_1(x) = b_0(x)-b_0(2x)$.
    Let $\chi\in C_c^\infty(\fa)$. For any $t\geq1$, $y\in\BR$, and $\nu\in\fa^*$, we have
    \begin{align*}
        (2^{iy}-1)\int_\fa \chi(H)\left(b_0\big(t^{n-r}\omega(H)\big) + \sum_{m=1}^{\infty} 2^{m(-\frac{r}{n-r}+iy)}  b_1\big(2^{-m}t^{n-r}\omega(H)\big)\right)e^{i\langle\nu,H\rangle} dH\lesssim t^{-r}.
    \end{align*}
    Here the implied constant depends on $\supp\chi$ and the sup-norms of $\chi$ and $b_0$ and their derivatives. 
\end{proposition}

\begin{proof}
    The result is trivial when $2^{iy}=1$ so we will assume $2^{iy}\neq1$. Moreover, the first term
    \[
   (2^{iy}-1)\int \chi(H)b_0\big(t^{n-r}\omega(H)\big)e^{i\langle\nu,H\rangle} dH\lesssim\int b_0\big(t^{n-r}\omega(H)\big)dH\lesssim t^{-r}
    \]
    satisfies the desired bound by the volume estimate \eqref{eq: volume estimates}. By the compact support of $\chi$, there exists a positive constant $C$ which depends on the size of the support of $\chi$, such that $ b_1\big(2^{-m}t^{n-r}\omega(H)\big)=0$ if $m>(n-r)\log_2t+C$, so the sum we are considering is finite, and it suffices to prove that
    \begin{align}\label{eq: L2 to L2 integral bound}
        (2^{iy}-1)\sum_{m=1}^M2^{m(-\frac{r}{n-r}+iy)}\int_\fa \chi(H)    b_1\big(2^{-m}t^{n-r}\omega(H)\big)e^{i\langle\nu,H\rangle} dH\lesssim t^{-r}.
    \end{align}
    Here $M=(n-r)\log_2t+C$ and we choose a suitable $C$ such that $M$ is an integer.
    
    Given $\nu\in\fa^*$ and $t\geq1$, we let $m_0$ be the largest integer satisfying that
    \[
    1\leq m_0\leq (n-r)\log_2t+C\quad\text{ and }\quad\left\| 2^{\frac{m_0}{n-r}}\nu\right\|\leq t.
    \]
    If such an integer doesn't exist, we set $m_0=0$.
    We bound the sum for $m> m_0$ using the Fourier decay estimate. Recall that Lemma \ref{lem: bd for hm via IBP} gives that
    \begin{align}\label{eq: fourier decay for hm}
        2^{m(-\frac{r}{n-r}+iy)} \int_\fa \chi(H)  b_1\big(2^{-m}t^{n-r}\omega(H)\big)e^{i\langle\nu,H\rangle} dH\lesssim  t^{-r}\left(1+\|t^{-1}2^{\frac{m}{n-r}}  \nu\|\right)^{-\delta},
    \end{align}
    for some $\delta>0$. Then summing \eqref{eq: fourier decay for hm} for $m>m_0$ gives
    \begin{align}\label{eq: L2L2 small m estimate}
    \begin{split}
       (2^{iy}-1)\sum_{m_0< m\leq (n-r)\log_2t+C}2^{m(-\frac{r}{n-r}+iy)}\int_\fa \chi(H)    b_1\big(2^{-m}t^{n-r}\omega(H)\big)e^{i\langle\nu,H\rangle} dH\\
        \lesssim t^{-r}\sum_{m>m_0} \|t^{-1}2^{\frac{m}{n-r}}\nu\|^{-\delta} \lesssim t^{-r}.
    \end{split}
    \end{align}

    To bound the remaining sum (which is empty if $m_0=0$), we will use a partial summation trick.
    For any $1\leq m\leq m_0$, by a change of variables $H\mapsto 2^{\frac{m}{n-r}}t^{-1}H$, we have
    \begin{align*}
         \int_\fa \chi(H)  b_1\big(2^{-m}t^{n-r}\omega(H)\big)e^{i\langle\nu,H\rangle} dH=2^{\frac{mr}{n-r}}t^{-r} \int_\fa \chi(2^{\frac{m}{n-r}}t^{-1}H)  b_1(\omega(H))e^{i2^{\frac{m}{n-r}}t^{-1}\langle\nu,H\rangle} dH.
    \end{align*}
    The sum we want to bound is
    \begin{align*}
         &(2^{iy}-1)\sum_{m=1}^{m_0}2^{m(-\frac{r}{n-r}+iy)}\int_\fa \chi(H)    b_1\big(2^{-m}t^{n-r}\omega(H)\big)e^{i\langle\nu,H\rangle} dH\\
         =&t^{-r}\sum_{m=1}^{m_0}(2^{i(m+1)y}-2^{imy})  \int_\fa \chi(2^{\frac{m}{n-r}}t^{-1}H)  b_1(\omega(H))e^{i2^{\frac{m}{n-r}}t^{-1}\langle\nu,H\rangle} dH.
    \end{align*}
    By a partial summation, the above is equal to 
    \begin{align}
         &t^{-r}2^{i(m_0+1)y}\int_\fa \chi( 2^{\frac{m_0}{n-r}}t^{-1}H)b_1\big(\omega(H)\big)e^{i 2^{\frac{m_0}{n-r}}t^{-1}\langle \nu,H\rangle} dH\notag\\
         -&t^{-r}2^{iy}\int_\fa \chi( 2^{\frac{1}{n-r}}t^{-1}H)b_1\big(\omega(H)\big)e^{i 2^{\frac{1}{n-r}}t^{-1}\langle \nu,H\rangle} dH+\notag\\
         &t^{-r}\sum_{m=2}^{m_0}2^{imy}\int_\fa b_1\big(\omega(H)\big) \left(\chi( 2^{\frac{m-1}{n-r}}t^{-1}H)e^{i 2^{\frac{m-1}{n-r}}t^{-1}\langle \nu,H\rangle}-  \chi( 2^{\frac{m}{n-r}}t^{-1}H)e^{i 2^{\frac{m}{n-r}}t^{-1}\langle \nu,H\rangle}\right) dH.\label{eq: L2L2 partial sum}
    \end{align}
    Using Proposition \ref{prop: cle volume estimate}, the first two terms above are clearly $\lesssim t^{-r}$. It remains to bound the sum \eqref{eq: L2L2 partial sum}. We first notice that this sum is trivial unless $m_0\geq2$. Therefore, by our definition for $m_0$, if $m_0\geq2$, we have for any $2\leq m \leq m_0$
    \begin{align}\label{eq: assumption on m mu}
       t^{-1}2^{\frac{m}{n-r}} (1+\|\nu\|)\lesssim 1,
    \end{align}
    and we will assume these in the rest of the proof. We claim that
    \begin{align}\label{eq: L2L2 difference of exps}
    \begin{split}
        \chi( 2^{\frac{m-1}{n-r}}t^{-1}H)e^{i 2^{\frac{m-1}{n-r}}t^{-1}\langle \nu,H\rangle}-  \chi( 2^{\frac{m}{n-r}}t^{-1}H)e^{i 2^{\frac{m}{n-r}}t^{-1}\langle \nu,H\rangle}\\
        \lesssim\min\big\{ 1,\; t^{-1}2^{\frac{m}{n-r}}(1+\|\nu\|)\|H\|\big\}.
    \end{split}
    \end{align}
    The bound $\lesssim1$ is trivial. To derive the second bound we consider the function $F_{\nu,H}\in C^\infty(0,\infty)$ defined by $F_{\nu,H}(\xi) = \chi(\xi H)e^{i\xi\langle\nu,H\rangle}$. It may be seen that
    \begin{align*}
        F_{\nu,H}'(\xi) \lesssim \|H\|+|\langle\nu,H\rangle| \lesssim \|H\|(1+\|\nu\|).
    \end{align*}
    By the mean value theorem, there exists some $2^{\frac{m-1}{n-r}}t^{-1}\leq \xi\leq2^{\frac{m}{n-r}}t^{-1}$ so that
    \begin{align*}
         \chi( 2^{\frac{m-1}{n-r}}t^{-1}H)e^{i 2^{\frac{m-1}{n-r}}t^{-1}\langle \nu,H\rangle}-  \chi( 2^{\frac{m}{n-r}}t^{-1}H)e^{i 2^{\frac{m}{n-r}}t^{-1}\langle \nu,H\rangle}&\lesssim F_{\mu,H}'(\xi)\left(  2^{\frac{m}{n-r}} -2^{\frac{m-1}{n-r}} \right)t^{-1}\\
         &\lesssim \|H\|(1+\|\nu\|)2^{\frac{m}{n-r}}t^{-1},
    \end{align*}
    so we obtain \eqref{eq: L2L2 difference of exps}. Therefore, by \eqref{eq: assumption on m mu} and \eqref{eq: L2L2 difference of exps}, we have
    \begin{align*}
        &\sum_{m=2}^{m_0}2^{imy}\left(\chi( 2^{\frac{m-1}{n-r}}t^{-1}H)e^{i 2^{\frac{m-1}{n-r}}t^{-1}\langle \nu,H\rangle}-  \chi( 2^{\frac{m}{n-r}}t^{-1}H)e^{i 2^{\frac{m}{n-r}}t^{-1}\langle \nu,H\rangle}\right) \\
        \lesssim& \sum_{m=2}^{m_0}   \min\big\{ 1,\; t^{-1}2^{\frac{m}{n-r}}(1+\|\nu\|)\|H\|\big\}\\
        =& \sum_{m=0}^{m_0-2}\min\big\{ 1,\; t^{-1}2^{\frac{m_0-m}{n-r}}(1+\|\nu\|)\|H\|\big\}\\
        \lesssim& \sum_{m=0}^{m_0-2}\min\big\{ 1,\; 2^{-\frac{m}{n-r}}\|H\|\big\}\\
        \lesssim& \sum_{0\leq m \leq (n-r)\log_2(1+\|H\|)} 1 + \sum_{ (n-r)\log_2(1+\|H\|)< m \leq m_0-2} 2^{-\frac{m}{n-r}}\|H\|\\
        \lesssim& \log (1+\|H\|) + \frac{\|H\|}{1+\|H\|}.
    \end{align*}
    so \eqref{eq: L2L2 partial sum} is bounded by
    \[
    t^{-r}\int_\fa b_1\big(\omega(H)\big)\Big(\log(1+\|H\|)+1\Big) dH\lesssim t^{-r},
    \]
    where the integral is finite because of Proposition \ref{prop: cle volume estimate} and Lemma \ref{lem: finite log integral}.
\end{proof}

\section{Spectral projector}
In this section we prove Theorem \ref{thm: main 2}.
Recall that we take a cocompact lattice $\Gamma\subset G$, and let $X=\Gamma\backslash S$ be the locally symmetric space. The Hilbert space $L^2(X)$ has an orthonormal basis $\{\psi_\mu\}_\mu$, where $\psi_\mu$'s are $L^2$-normalized Maass forms of spectral parameters $\mu$. 
We fix a compact set $B^*\subset \fa_+^*$. Given $\lambda\in B^*$ and $t\geq1$ sufficiently large, we can define the spectral projection operator (onto the ball of radius 1 around $t\lambda$) by \eqref{eq: spectral projector defn}.

\subsection{Smooth modification}
We first show that we are allowed to modify the spectral projector to have a smooth symbol which is suitable for the pre-trace formula.
Given a bi-$K$-invariant test function $k\in C_c^\infty(G\sslash K)$, it forms an automorphic kernel
\(
K(x,y)=\sum_{\gamma\in\Gamma} k(x^{-1}\gamma y),
\) on $X$.
Moreover, by Selberg \cite{selberg1956harmonic}, it satisfies for any Maass form $\psi_\mu$ of spectral parameter $\mu$
    \[
    \int_{\Gamma\backslash G/K} K(x,y)\psi_\mu(y) dy = \sH k(-\mu) \psi_\mu(x).
    \]

We fix one $h^0\in\mathcal{PW}(\fa^*)^W$ that is real-valued on $\fa^*$ and satisfies
\begin{align}\label{eq: condition for h0}
    h^0(0)=1, \quad\text{ and } \quad h^0(\nu)\geq1 \text{ in a ball of radius 1 at the origin}.
\end{align}
Given $\lambda\in B^*$ and $t\geq1$ we let
\begin{align}\label{eq: defn for h0t}
    h_{t\lambda}^0(\nu) =\sum_{w\in W}h^0(w\nu+t\lambda)\in \mathcal{PW}(\fa^*)^W.
\end{align}

It is clear that $h_{t\lambda}^0(-t\lambda)\gtrsim1$ and $h_{t\lambda}^0(\nu)$ rapidly decays if $\nu$ is away  from $-tW\lambda$. Let $g^0\in C_c^\infty(\fa)^W$ be the inverse Fourier transform of $h^0$, i.e.,
\[
g^0(H) = \int_{\fa^*} h^0(\nu)e^{i\langle\nu,H\rangle} d\nu,\qquad h^0(\nu)=\int_{\fa} g^0(H)e^{-i\langle\nu,H\rangle} dH 
\]
by normalizing the measures correctly. We will write this Fourier transform by $\sF_\fa: C_c^\infty(\fa)^W\to \cP\cW(\fa^*)^W$, i.e., $h^0 = \sF_\fa{g^0}$.
Therefore,
\begin{align}\label{eq: g0 integral for h0t}
  h_{t\lambda}^0(\nu)=\int_\fa \left(\sum_{w\in W} g^0(H) e^{-it\langle w\lambda,H\rangle}\right)e^{-i\langle \nu ,H\rangle} dH.
\end{align}

Now we give an explicit construction of an auxiliary wider Paley--Wiener function on $\fa^*$. We fix an even bump function $\beta\in C_c^\infty(\fa^*)^W$ such that $0\le\beta\le1$ is real-valued, $\beta(\xi)=1$ for $\|\xi\|\le1$ and $\beta(\xi)=0$ for $\|\xi\|\geq2$. Similarly we fix $\beta'\in C_c^\infty(\fa)^W$ with the same properties as $\beta$. 
Fix a tiny constant $\epsilon_0>0$ such that $\epsilon_0$ is smaller than the distance between $B^*$ and the singular subset of $\fa^*$.
We define a bump function  $\beta^{\epsilon_0}_{t\lambda}\in C_c^\infty(\fa^*)^W$ by
\begin{align*}
    \beta^{\epsilon_0}_{t\lambda}(\nu) =\sum_{w\in W} \beta\left(\frac{w\nu+t\lambda}{\epsilon_0t}  \right),
\end{align*}
so
\[
    \sF_\fa^{-1}\beta^{\epsilon_0}_{t\lambda}(H)= (\epsilon_0t)^r(\sF_\fa^{-1}\beta) (\epsilon_0 t H)\sum_{w\in W} e^{-it\langle w\lambda,H\rangle} .
\]
We define a function $\widetilde \beta^{\epsilon_0}_{t\lambda}$ by its Fourier inversion so that
\begin{align*}
    \sF_\fa^{-1}\widetilde \beta^{\epsilon_0}_{t\lambda}(H) &=\beta'(H)(\sF_\fa^{-1}\beta^{\epsilon_0}_{t\lambda})(H)\\
    &=(\epsilon_0t)^r\beta'(H)(\sF_\fa^{-1}\beta) (\epsilon_0 t H)\sum_{w\in W} e^{-it\langle w\lambda,H\rangle}.
\end{align*}
Note that by definition $\sF_\fa^{-1}\widetilde \beta^{\epsilon_0}_{t\lambda}\in C_c^\infty(\fa)^W$ so $\widetilde \beta^{\epsilon_0}_{t\lambda}\in\cP\cW(\fa^*)^W$.

\begin{lemma}\label{lem: PW tail estimate}
    Let ${ \beta^{\epsilon_0}_{t\lambda}}$ and $\widetilde \beta^{\epsilon_0}_{t\lambda}$ be defined as above and let $\beta^{\rm tail} =\widetilde \beta^{\epsilon_0}_{t\lambda} - { \beta^{\epsilon_0}_{t\lambda}}$. Then for any $N\geq1$ and $\nu\in\fa^*$ we have
    \[
    \beta^{\rm tail}(\nu) \lesssim \sum_{w\in W}(1+\|w\nu+t\lambda\|)^{-N} t^{-N},
    \]
    where the implied constant depends only on $N$ and $\epsilon_0$.
\end{lemma}
\begin{proof}
    It may be seen that
    \begin{align*}
        \beta^{\rm tail}(\nu) =(\epsilon_0t)^r\sum_{w\in W}\int_\fa(\beta'(H)-1)(\sF_\fa^{-1}\beta) (\epsilon_0 t H)e^{-i\langle w\nu+t\lambda,H\rangle} dH,
    \end{align*}
    so it suffices to show that for arbitrary $\nu\in\fa^*$ and $N\geq1$
    \begin{align*}
        \int_\fa(1-\beta'(H))(\sF_\fa^{-1}\beta) (\epsilon_0 t H)e^{-i\langle \nu,H\rangle} dH\lesssim_N (1+\|\nu\|)^{-N} t^{-N}.
    \end{align*}
    This is immediate after integrating by parts and applying the Paley--Wiener condition for $\sF_\fa^{-1}\beta$.
\end{proof}

Now we can define two test functions $k_{t\lambda}^0$ and $\widetilde k_{t\lambda}^0$ in $C_c^\infty(G\sslash K)$ by fixing their Harish-Chandra transforms:
\begin{align}\label{eq: smooth test functions defn}
   \sH k_{t\lambda}^0 =h^0_{t\lambda},\quad\text{ and }\quad \sH \widetilde k_{t\lambda}^0 = \widetilde \beta^{\epsilon_0}_{t\lambda}h^0_{t\lambda}.
\end{align}
We define their associated automorphic kernels to be $K_{t\lambda}^0(x,y)=\sum_{\gamma\in\Gamma} k_{t\lambda}^0(x^{-1}\gamma y)$ and $\widetilde K_{t\lambda}^0(x,y)=\sum_{\gamma\in\Gamma} \widetilde k_{t\lambda}^0(x^{-1}\gamma y)$.

\begin{lemma}\label{lem: equivalent upper bound}
    Let $\bP_{t\lambda}$ be the spectral projector defined in \eqref{eq: spectral projector defn}.
    For any $p\geq2$, we have
    \[
    \|\bP_{t\lambda}\|_{L^2\to L^p}\lesssim\|K^0_{t\lambda}\|_{L^2\to L^p}\quad\text{ and }\quad \|\bP_{t\lambda}\|_{L^2\to L^p}\lesssim\|\widetilde K^0_{t\lambda}\|_{L^2\to L^p}.
    \]
\end{lemma}
\begin{proof}
    Let $f\in L^2(X)$ with $f = \sum_{\mu} a_\mu \psi_\mu$, so
    \begin{align*}
        \bP_{t\lambda}f = \sum_{\|\mu-t\lambda\|\leq1}a_\mu\psi_\mu = \sum_{\|\mu-t\lambda\|\leq1}h_{t\lambda}^0(-\mu)\frac{a_\mu}{h_{t\lambda}^0(-\mu)}\psi_\mu = K_{t\lambda}^0\tilde{f},
    \end{align*}
    where
    \[
    \tilde{f} =  \sum_{\|\mu-t\lambda\|\leq1}\frac{a_\mu}{h_{t\lambda}^0(-\mu)}\psi_\mu.
    \]
    Thus we have
    \begin{align*}
        \|\tilde{f}\|^2_{L^2} =  \sum_{\|\mu-t\lambda\|\leq1}\left|\frac{a_\mu}{h_{t\lambda}^0(-\mu)}\right|^2 \lesssim \sum_{\|\mu-t\lambda\|\leq1}\left|{a_\mu}\right|^2 \leq\|{f}\|^2_{L^2}
    \end{align*}
    and 
    \begin{align*}
        \|\bP_{t\lambda}f\|_{L^p} = \|K_{t\lambda}^0\tilde{f}\|_{L^p} \leq \|K^0_{t\lambda}\|_{L^2\to L^p} \|\tilde{f}\|_{L^2}\lesssim\|K^0_{t\lambda}\|_{L^2\to L^p} \|{f}\|_{L^2}.
    \end{align*}
    This proves the first bound, and it is clear that the same proof gives the second bound as well.
\end{proof}

By the above lemma and a $TT^*$ argument, to prove \eqref{eq: thm main 2} it suffices to prove that $\widetilde K_{t\lambda}:=\widetilde K_{t\lambda}^0\circ (\widetilde K_{t\lambda}^0)^*$ has the following operator norm bound
\begin{align}\label{eq: TT* operator norm}
    \left\|\widetilde K_{t\lambda}\right\|_{L^{p_c'}\to L^{p_c}}\lesssim t^{2r/p_c}
\end{align}
where $p_c'$ is the dual exponent of  the critical exponent $p_c = \frac{2(n+r)}{n-r}$. It may be seen that one can take $\widetilde k_{t\lambda} = \widetilde k_{t\lambda}^0*\widetilde k_{t\lambda}^0$ so that $\widetilde K_{t\lambda}$ is associated with $\widetilde k_{t\lambda}$.  Note that $\widetilde K_{t\lambda}^0$ is actually self-adjoint because $\widetilde\beta_{t\lambda}^{\epsilon_0}h_{t\lambda}^0$ is real on $\fa^*$.
Let $h_{t\lambda} = (h_{t\lambda}^0)^2$,
so $\left(\widetilde \beta^{\epsilon_0}_{t\lambda}\right)^2h_{t\lambda} = \sH \widetilde k_{t\lambda}$.


\begin{lemma}\label{lem: Fourier of ht}
We let $g=g^0*g^0\in C_c^\infty(\fa)^W$. Then, for any $\nu\in\fa_\BC^*$,
    \[
    h_{t\lambda}(\nu)=\int_\fa \left(\sum_{w\in W} g(H) e^{-it\langle w\lambda,H\rangle}\right)e^{-i\langle \nu ,H\rangle} dH + \cE_{t\lambda}(\nu).
    \]
    Here $\cE_{t\lambda}\in\cP\cW(\fa^*)^W$ is an error term satisfying
    \begin{align}\label{eq: error estimates}
        \cE_{t\lambda}(\nu)\lesssim_N e^{O(\Im\nu)}\sum_{w\in W} (1+\|w(\Re\nu)+t\lambda\|)^{-N}t^{-N}.
    \end{align}
\end{lemma}
\begin{proof}
    We first calculate the inverse Fourier inversion $\sF_\fa^{-1}$ of $h_{t\lambda}=(h_{t\lambda}^0)^2$, which is given by the following convolution from \eqref{eq: g0 integral for h0t}
    \begin{align*}
        &\int_\fa \left(\sum_{w_1\in W} g^0(H-X) e^{-it\langle w_1\lambda,H-X\rangle}\right)\left(\sum_{w_2\in W} g^0(X) e^{-it\langle w_2\lambda,X\rangle}\right) dX\\
        =&\sum_{w_1,w_2}e^{-it\langle w_1\lambda ,H\rangle} \int_\fa g^0(H-X)g^0(X) e^{it\langle w_1\lambda-w_2\lambda,X\rangle} dX\\
         =&\sum_{w\in W}g(H)e^{-it\langle w\lambda ,H\rangle} +\sum_{w_1\neq w_2}e^{-it\langle w_1\lambda ,H\rangle} \int_\fa g^0(H-X)g^0(X) e^{it\langle w_1\lambda-w_2\lambda,X\rangle} dX.
    \end{align*}
    Therefore,
    \begin{align*}
        \cE_{t\lambda}(\nu)=\sum_{w_1\neq w_2}\int_\fa\int_\fa g^0(H-X)g^0(X) e^{it\langle w_1\lambda-w_2\lambda,X\rangle} e^{-i\langle\nu+tw_1\lambda,H\rangle} dX dH
    \end{align*}
    By integrating by parts with respect to $X$ and with respect to $H$, it is clear that \eqref{eq: error estimates} holds.
\end{proof}

\subsection{Holomorphic family of operators}\label{sec: complex family of projectors}

Let $b_0\in C_c^\infty(\BR)$ be a (nonnegative) even cutoff function that is equal to 1 on $[-1,1]$ and zero outside $[-2,2]$, and let $b_1(x) = b_0(x)-b_0(2x)$. We have the following dyadic decomposition for any $H\in\fa$,
\[
    1= b_0\big(t^{n-r}\omega(H)\big) + \sum_{m=1}^{\infty}  b_1\big(2^{-m}t^{n-r}\omega(H)\big).
\]
We let $g^0,h^0$ be as in the last section and let $g=g^0*g^0$ as in Lemma \ref{lem: Fourier of ht}.
By the compact support of $g$, the decomposition
\[
    g(H)= g(H)\left(b_0\big(t^{n-r}\omega(H)\big) + \sum_{m=1}^{\infty} b_1\big(2^{-m}t^{n-r}\omega(H)\big)\right)
\]
is actually a finite sum.

For a complex variable $z\in\BC$, we define, for $H\in\fa$,
\begin{align*}
    g(H;z) =(2^{z}-2^{-\frac{r}{n-r}})t^{-(n-r)z}g(H)\left(b_0\big(t^{n-r}\omega(H)\big) + \sum_{m=1}^{\infty} 2^{mz}b_1\big(2^{-m}t^{n-r}\omega(H)\big)\right).
\end{align*}
As we have observed, the above is still actually a finite sum.
For $\nu\in\fa^*_\BC$, we let
\begin{align*}
    h(\nu;z)=\int_\fa \left(\sum_{w\in W} g(H;z) e^{-it\langle w\lambda,H\rangle}\right)e^{-i\langle \nu ,H\rangle} dH.
\end{align*}
Note that for any $z\in\BC$, $g(\cdot;z)\in C_c^\infty(\fa)^W$ so $h(\cdot;z)\in \mathcal{PW}(\fa^*)^W$ form a family of functions depending holomorphically on $z$. We will write
\[
g_0(H;z)=(2^{z}-2^{-\frac{r}{n-r}})t^{-(n-r)z}g(H)b_0\big(t^{n-r}\omega(H)\big) ,\]
and
\[g_m(H;z)=(2^{z}-2^{-\frac{r}{n-r}})t^{-(n-r)z}g(H) 2^{mz}b_1\big(2^{-m}t^{n-r}\omega(H)\big),\qquad m\geq1,
\]
so that $g(H;z) = \sum_{m=0}^{\infty} g_m(H;z)$. Correspondingly we write $h(\nu;z) = \sum_{m=0}^{\infty} h_m(\nu;z)$ with
\[
h_m(\nu;z) = \int_\fa \left(\sum_{w\in W} g_m(H;z) e^{-it\langle w\lambda,H\rangle}\right)e^{-i\langle \nu ,H\rangle} dH.
\]

Moreover, when $z=0$, we have $h(\nu;0)=(1-2^{-\frac{r}{n-r}})h_{t\lambda}(\nu)+O_N(t^{-N})$ by Lemma \ref{lem: Fourier of ht}. We define a holomorphic family of operators $\{T_z\}_{z\in\BC}$ on $X$ by the kernels
\[
    T_z(x,y)=\sum_{\gamma\in\Gamma}k(x^{-1}\gamma y;z),
\]
with
\[
    \big(\sH k(\cdot;z)\big)(\nu) = \left(\widetilde \beta^{\epsilon_0}_{t\lambda}(\nu)\right)^2(h(\nu;z)+(1-2^{-\frac{r}{n-r}})\cE_{t\lambda}(\nu)).
\]

\begin{proposition}\label{prop 6.4}
    When $\Re z \geq0$, the operator norm of $T_z:L^1(X)\to L^\infty(X)$ is $\lesssim \max\{  t^{(n-r)(1-\Re z)}, t^{\frac{n-r}{2}}\}$. When $\Re z = -\frac{r}{n-r}$, the operator norm of $T_z:L^2(X)\to L^2(X)$ is $\lesssim 1$.
\end{proposition}

When $z=0$, by Lemma \ref{lem: Fourier of ht}, the operator $T_0$ is equal to $(1-2^{-\frac{r}{n-r}})\widetilde K_{t\lambda}$.
Therefore, \eqref{eq: TT* operator norm} follows from Stein's complex interpolation \cite{Ste56} for the holomorphic family of operators $\{T_z\}$ in the vertical strip $-\frac{r}{n-r}\leq\Re z\leq \frac{1}{2}$, and thus the critical exponent estimate \eqref{eq: thm main 2} follows.

\begin{lemma}\label{lem: bd for spectral integral}
    Let $g$ be as above and let $B\subset\fa$ be a compact subset.
    We can choose $\epsilon_0$ to be sufficiently small so that the following holds.
    Given $X\in B$ and $H\in\supp(g)$, for any $N\geq1$,
    \begin{align*}
        \int_{\fa^*}e^{-i\langle \nu,H\rangle}\left(\widetilde \beta^{\epsilon_0}_{t\lambda}(\nu)\right)^2\varphi_\nu(e^X) d\mu(\nu)\lesssim t^{n}\left(\prod_{\alpha\in\Delta^+}(1+t|\alpha(X)|)^{-m_\alpha/2} \right)\sum_{w\in W}\left( 1+t\|H-wX\| \right)^{-N},
    \end{align*}
    where the implied constant depends only on $N$, $B$, $g$ and $\epsilon_0$.
\end{lemma}
\begin{proof}
    By Lemma \ref{lem: PW tail estimate}, we have
    \begin{align*}
        \left(\widetilde \beta^{\epsilon_0}_{t\lambda}(\nu)\right)^2=\left(\sum_{w\in W} \beta\left(\frac{w\nu+t\lambda}{\epsilon_0t} \right)+O_N\left((1+\|w\nu+t\lambda\|)^{-N} t^{-N}\right)  \right)^2 \\
        =\sum_{w\in W} \beta\left(\frac{w\nu+t\lambda}{\epsilon_0t}  \right)^2+O_N\left(\sum_{w\in W}(1+\|w\nu+t\lambda\|)^{-N} t^{-N}\right),
    \end{align*}
    so by the change of variables $\nu\mapsto t\nu$, the integral considered in the lemma is
    \begin{align}\label{eq: 6.2}
        \sum_{w\in W}\int_{\fa^*}e^{-it\langle \nu,H\rangle} \beta\left(\frac{w\nu+\lambda}{\epsilon_0}\right)^2\varphi_{t\nu}(e^X) d\mu(t\nu) + O_N(t^{-N}).
    \end{align}
    By choosing $\epsilon_0$ sufficiently small, there exists $\Omega^* \in \fa^*_{+}$ such that $\nu$ in \eqref{eq: 6.2} is supported in $W\Omega^*$. Using \eqref{eq: HC integral for varphi} and applying Theorem \ref{thm: Weyl Harish-Chandra symbol} to $\beta\left(\frac{\nu+\lambda}{\varepsilon_0}\right)^2\varphi_{t\nu}(e^X)$, there exist coefficients $a_{w,t}(X,\nu)$ which are smooth in $\nu$ and satisfy the derivative bounds \eqref{eq:symbol}, so that the main term of \eqref{eq: 6.2} is equal to
    \begin{align*}
        \sum_{w,w'\in W}\int_{\fa^*}a_{w',t}(X,\nu)e^{it\langle \nu,-wH+w'X\rangle} d\mu(t\nu).
    \end{align*}
    Write $d\mu(t\nu) =M(t\nu) t^r d\nu$ where $M(t\nu) = \abs{\bc(t\nu)}^{-2}$ is the Plancherel density.  By the Gindikin--Karpelevic product formula and Stirling's formula, the factor $M(t\nu)$ has size $\lesssim t^{n-r}$, depending only on $B^*$, and the same holds for all derivatives in $\nu$, using the regular assumption on $B^*$ and the Cauchy integral formula. Therefore \eqref{eq: 6.2} is of the form
    \begin{align*}
        \sum_{w,w'\in W}\int_{\fa^*}\Big(t^{r}a_{w',t}(X,\nu)M(t\nu)\Big)e^{it\langle \nu,-wH+w'X\rangle} d\nu,
    \end{align*}
    with the amplitude $t^r a_{w',t}(X,\nu)M(t\nu)$ and its derivatives in $\nu$ bounded by
    \[
    \lesssim t^{n}\prod_{\alpha\in\Delta^+}(1+t|\alpha(X)|)^{-m_\alpha/2},
    \]
    so we finish the proof by integrating by parts $N$ times.
\end{proof}

\begin{proof}[Proof of Proposition \ref{prop 6.4}]
By Young's inequality, $\|T_z\|_{L^1\to L^\infty}$ is dominated by the sup-norm of $k(\cdot;z)$. Moreover, by the bi-$K$-invariance, we need to show that, for $X\in\overline{\fa^+}$ and $\Re z \geq0$,
\begin{multline*}
    k(e^X;z)=\frac{1}{|W|}\int_{\fa^*} \left(\sum_{m=0}^{\infty}\int_\fa \left(\sum_{w\in W} g_m(H;z) e^{-it\langle w\lambda,H\rangle}\right)e^{-i\langle \nu ,H\rangle} dH +(1-2^{-\frac{r}{n-r}})\cE_{t\lambda}(\nu) \right)\\
    \left(\widetilde \beta^{\epsilon_0}_{t\lambda}(\nu)\right)^2\varphi_\nu(e^X) d\mu(\nu)
\end{multline*}
is $\lesssim \max\{  t^{(n-r)(1-\Re z)}, t^{\frac{n-r}{2}}\}$. By Lemma \ref{lem: Fourier of ht}, it is clear that the contribution of the term with $\cE_{t\lambda}$ 
\begin{align*}
    \frac{1}{|W|}\int_{\fa^*} (1-2^{-\frac{r}{n-r}})\cE_{t\lambda}(\nu) \left(\widetilde \beta^{\epsilon_0}_{t\lambda}(\nu)\right)^2\varphi_\nu(e^X) d\mu(\nu)\lesssim_N t^{-N}
\end{align*}
is negligible. Therefore, to bound $\|T_z\|_{L^1\to L^\infty}$, it suffices to prove
\begin{align}
    \int_\fa\int_{\fa^*}  \sum_{m=0}^{\infty}g_m(H;z) e^{-i\langle \nu+t\lambda,H\rangle}\left(\widetilde \beta^{\epsilon_0}_{t\lambda}(\nu)\right)^2\varphi_\nu(e^X)  d\mu(\nu)dH\lesssim  \max\{  t^{(n-r)(1-\Re z)}, t^{\frac{n-r}{2}}\}\label{eq: L1Linfinity integral bound}.
\end{align}
By the Paley--Wiener theorem, $k(\cdot;z)$ is supported in a compact set, so we will fix a compact subset $B\subset\overline{\fa^+}$ and assume that $X\in B$.

    We first notice that for each $H\in\supp g$, there are at most $O(1)$ many $m$'s so that $H\in\supp g_m$, so for $\Re z \geq0$, one has
    \begin{align*}
        \sum_{m=0}^{\infty}g_m(H;z) \lesssim t^{-(n-r)\Re z } + |\omega(H)|^{\Re z } \lesssim \prod_{\alpha\in\Delta^+} (t^{-1}+|\alpha(H)|)^{m_\alpha \Re z }.
    \end{align*}
    For each positive root $\alpha\in\Delta^+$
    \begin{align*}
        t^{-1}+|\alpha(H)|\lesssim (t^{-1}+|\alpha(wX)|)(1+t\|H-wX\|),
    \end{align*}
    so taking the product gives
    \begin{align*}
        \sum_{m=0}^{\infty}g_m(H;z)  \lesssim &\prod_{\alpha\in\Delta^+} (t^{-1}+|\alpha(H)|)^{m_\alpha\Re z } \\
        \lesssim &(1+t\|H-wX\|)^{(n-r)\Re z }\prod_{\alpha\in\Delta^+} (t^{-1}+|\alpha(wX)|)^{m_\alpha \Re z }\\
        =&t^{-(n-r)\Re z }(1+t\|H-wX\|)^{(n-r)\Re z }\prod_{\alpha\in\Delta^+} (1+t|\alpha(X)|)^{m_\alpha\Re z }.
    \end{align*}
    We combine this with Lemma \ref{lem: bd for spectral integral} to obtain
    \begin{align*}
        &\int_\fa\int_{\fa^*}  \sum_{m=0}^{\infty}g_m(H;z) e^{-i\langle \nu+t\lambda,H\rangle}\left(\widetilde \beta^{\epsilon_0}_{t\lambda}(\nu)\right)^2\varphi_\nu(e^X)  d\mu(\nu)dH\\
        \lesssim & \sum_{w\in W} \int_{\fa} t^n t^{-(n-r)\Re z}\prod_{\alpha\in\Delta^+} (1+t|\alpha(X)|)^{m_\alpha(\Re z-\frac{1}{2}) } (1+t\|H-wX\|)^{-N} dH\\
        \lesssim & t^{(n-r)(1-\Re z)}\prod_{\alpha\in\Delta^+} (1+t|\alpha(X)|)^{m_\alpha(\Re z-\frac{1}{2}) }\\
        \lesssim & t^{(n-r)(1-\Re z)} \max\{  1, t^{(n-r)(\Re z-\frac{1}{2})}\},
    \end{align*}
    which proves \eqref{eq: L1Linfinity integral bound}.

Now we prove the estimate when $\Re z=-\frac{r}{n-r}$.
We note that $\|T_z\|_{L^2\to L^2}$ is equal to $$\sup \left\{\left|\left(\widetilde \beta^{\epsilon_0}_{t\lambda}(-\nu)\right)^2(h(-\nu;z)+(1-2^{-\frac{r}{n-r}})\cE_{t\lambda}(-\nu))\right|:\nu\in\cS \right\},$$
where $\cS\subset\fa^*_\BC/W$ is the set of spectral parameters of the joint eigenfunctions in $L^2(X)$. If $\nu\in\cS$, then either $\nu\in\fa^*$ or $\Re(\nu)$ is singular and $\|\Im(\nu)\|\leq\|\rho\|$. By Lemma \ref{lem: Fourier of ht} we can clearly ignore the $\cE_{t\lambda}$ term, and it suffices to prove that
    \begin{align}\label{eq: int 6.8}
       (2^{z}-2^{-\frac{r}{n-r}})t^{-(n-r)z} \int_\fa g(H)\left(b_0\big(t^{n-r}\omega(H)\big) + \sum_{m=1}^{\infty} 2^{mz}b_1\big(2^{-m}t^{n-r}\omega(H)\big)\right)e^{i\langle \mu,H\rangle} dH\lesssim 1
    \end{align}
    uniformly for $\mu\in\fa^*_\BC$ with $\Im\mu\lesssim1$. By modifying $g\in C_c^\infty(\fa)$, we will further assume that $\mu \in\fa^*$. If $\Re z=-\frac{r}{n-r}$, by writing $z=-\frac{r}{n-r}+iy$ with $y\in\BR$, the integral from \eqref{eq: int 6.8} becomes
    \[
    2^{-\frac{r}{n-r}}t^{r-i(n-r)y} (2^{iy}-1)\int_\fa g(H)\left(b_0\big(t^{n-r}\omega(H)\big) + \sum_{m=1}^{\infty}2^{m(-\frac{r}{n-r}+iy)}  b_1\big(2^{-m}t^{n-r}\omega(H)\big)\right)e^{i\langle \mu,H\rangle} dH.
    \]
    Then by Proposition \ref{prop: symbol estimates}, \eqref{eq: int 6.8} holds.
\end{proof}

\subsection{Sharpness}\label{sec: sharpness}
In this section, we prove the sharpness of the bounds \eqref{eq: thm main 2 general p}. 

Note that we can show the sharpness of
\[
\|\bP^Q_{t\lambda} f\|_{L^{p}(X)}\lesssim t^{r\delta(n/r,p)} \|f\|_{L^2(X)}
\]
instead with a fixed constant $Q\geq 1$.
When $p\geq p_c$, point-concentrated quasi-modes would saturate the bounds \eqref{eq: thm main 2 general p}.
By the local Weyl law (see e.g. \cite[Theorem 8.5]{DKV79}), and by picking the constant $Q$ sufficiently large, we can find $Z_{t\lambda}\in C^\infty(X)$ such that $Z_{t\lambda}=\bP^Q_{t\lambda}Z_{t\lambda}$, $\|Z_{t\lambda} \|_{L^2(X)}=1$ and $\|Z_{t\lambda} \|_{L^\infty(X)} = |Z_{t\lambda} (x)|\gtrsim t^{\frac{n-r}{2}}$ for some $x\in X$. 
Let $\Delta$ be the Laplace--Beltrami operator on $X$. By \cite[Chapter II, Corollary 5.20]{helgason1984groups}, for any Maass form $\psi_\mu$ with tempered spectral parameter $\mu\in\fa^*$ we have
\[
    \Delta \psi_\mu + (\|\mu\|^2+\|\rho\|^2)\psi_\mu=0.
\]
Let $P=\sqrt{-\Delta-\|\rho\|^2}$, and fix $\eta \in C_0^\infty((1/4,4))$ such that $\eta=1$ on $[1/2,2]$.
Then by the definition of $\bP_{t\lambda}$
in \eqref{eq: spectral projector defn},
it is not hard to check that $\eta(\frac{P}{t\|\lambda\|})\bP^Q_{t\lambda}=\bP^Q_{t\lambda}$. 
Indeed, if $\|\mu-t\lambda\|\leq Q$, then
$$
\big|\|\mu\|-t\|\lambda\|\big|
\leq \|\mu-t\lambda\|
\leq Q,
$$
and hence $\|\mu\|/(t\|\lambda\|)\in[1/2,2]$ for sufficiently large $t$.
We also have the Bernstein type inequality
\begin{equation}\label{Berstein}
   \|\eta(P/t\|\lambda\|)\|_{L^p\rightarrow L^q}\lesssim t^{\frac{n}{p}-\frac{n}{q}},\qquad 1\le p\le q\le \infty.
\end{equation}
See e.g.,
 \cite[Theorem 4.3.1]{SFIOII} and its proof for more details. 
Therefore, for the function $Z_{t\lambda} $ constructed above, which lies in the range of $\bP^Q_{t\lambda}$,
$$
\|Z_{t\lambda} \|_{L^\infty(X)}=\left\|\eta(P/t\|\lambda\|)Z_{t\lambda} 
\right\|_{L^\infty(X)}
\lesssim
t^{n/p}\|Z_{t\lambda} \|_{L^p(X)}.
$$
Combining this with the lower bound $\|Z_{t\lambda} \|_{L^\infty(X)}\gtrsim t^{\frac{n-r}{2}}$,
we obtain
\[
\|Z_{t\lambda} \|_{L^p(X)}
\gtrsim
t^{\frac{n-r}{2}-\frac np},
\qquad p\geq2 .
\]
This proves the sharpness for $p\geq p_c$.

To prove the sharpness in the range $2\leq p\leq p_c$, we need to reduce this to proving the sharpness of the corresponding bounds for the approximate spectral projector $K^0_{t\lambda}$, which is defined by \eqref{eq: defn for h0t} and \eqref{eq: smooth test functions defn}.

\begin{proposition}
    Let $B^*\subset \fa^*_+$ be a compact subset. Assume that $B^*$ has nonempty interior. Suppose that for some $p\geq2$ there exists an improved bound
    \begin{align}\label{eq: improved bound for genuine projector}
        \|\bP_{t\mu}\|_{L^2\to L^p} = o(t^{r\delta(n/r,p)})
    \end{align}
    for any $t\geq1$ and $\mu\in B^*$. Then for any $\lambda$ in the interior of $B^*$ we have
    \begin{align*}
        \|K^0_{t\lambda}\|_{L^2\to L^p} = o(t^{r\delta(n/r,p)})
    \end{align*}
    for any $t\geq1$.
\end{proposition}
\begin{proof}
    We first recall that we let $\cS\subset\fa^*_\BC/W$ be the set of spectral parameters in $L^2(X)$. If $\nu\in\cS$, then either $\nu\in\fa^*$ or $\Re(\nu)$ is singular and $\|\Im(\nu)\|\leq\|\rho\|$.
    Recall that from our construction \eqref{eq: condition for h0} and $\eqref{eq: defn for h0t}$, we have the following Paley--Wiener property for $h^0_{t\lambda}$. There exists $R>0$ such that for any $\nu\in\fa_\BC^*$
    \begin{align}\label{eq: PW property for h0t}
        h^0_{t\lambda}(\nu)\lesssim_{N} e^{R\|\Im \nu\|} \sum_{w\in W}(1+\|w\Re\nu+t\lambda\|)^{-N}.
    \end{align}
    Let $\epsilon_1>0$ be a sufficiently small constant, which depends only on $B^*$ and $\lambda$, such that the ball centered at $\lambda$ of radius $\epsilon_1$ is contained in $B^*$. By \eqref{eq: PW property for h0t}, for any $\nu\in\cS$ that $\|\Re(\nu)+tW\lambda\|\geq \epsilon_1t$, we have
    \[
    h_{t\lambda}^0(\nu)\lesssim_N t^{-N}.
    \]
    Let $\{ \psi_\mu\}$ be an orthonormal basis of Maass forms for $L^2(X)$.
    It suffices to consider those $f\in L^2(X)$ with the spectral expansion
    \begin{align}\label{eq: expansion of f in t ball}
        f=\sum_{\|\mu-t\lambda\|\leq \epsilon_1t} a_\mu \psi_\mu
    \end{align}
    and prove that for such $f$ we have
    \begin{align}\label{eq: improved K0 bound}
        \|K^0_{t\lambda} f\|_{L^p(X)} = o\Big(t^{r\delta(n/r,p)} \| f\|_{L^2(X)}\Big).
    \end{align}
    We can find a lattice in $\fa^*$ such that there are $O(t^r)$ many lattice points $\{t\mu_i\}$ whose radius 1 balls cover the ball $\{\mu:\|\mu-t\lambda\|\leq \epsilon_1t \}$. Moreover by our assumption on $\epsilon_1$, we have $\{\mu_i\}\subset B^*$. Therefore, for $f$ as in \eqref{eq: expansion of f in t ball}, we have
    \begin{align*}
        \|K_{t\lambda}^0 f\|_{L^p} = \Big\|\sum_{\|\mu-t\lambda\|\leq \epsilon_1t} h_{t\lambda}^0(-\mu) a_\mu \psi_\mu\Big\|_{L^p}
        \lesssim 
        \sum_{\mu_i}\|\bP_{t\mu_i}\|_{L^2\to L^p}
        \Big \| \sum_{\|\mu-t\mu_i\|\leq 1}h_{t\lambda}^0(-\mu) a_\mu \psi_\mu\Big\|_{L^2}.
    \end{align*}
    From\eqref{eq: improved bound for genuine projector}, we have $\|\bP_{t\mu_i}\|_{L^2\to L^p} = o(t^{r\delta(n/r,p)})$ for each $\mu_i$. Then by applying \eqref{eq: PW property for h0t} and the Cauchy--Schwarz inequality we obtain

    \begin{align*}
        &\sum_{\mu_i}
        \Big \| \sum_{\|\mu-t\mu_i\|\leq 1}h_{t\lambda}^0(-\mu) a_\mu \psi_\mu\Big\|_{L^2}
        \\
        \lesssim& \sum_{\mu_i}\left( \sum_{\mu\in\cS:\|\mu-t\mu_i\|\leq 1}|h_{t\lambda}^0(-\mu) a_\mu |^2\right)^{1/2}
        \\
        \lesssim&_N \sum_{\mu_i}(1+t\|\mu_i-\lambda \|)^{-N}\left( \sum_{\mu\in\cS:\|\mu-t\mu_i\|\leq 1}|a_\mu |^2\right)^{1/2}
        \\
        \leq& \left(\sum_{\mu_i}(1+t\|\mu_i-\lambda \|)^{-2N}\right)^{1/2}\left(\sum_{\mu_i}\sum_{\mu\in\cS:\|\mu-t\mu_i\|\leq 1}|a_\mu |^2\right)^{1/2}
        \\
        \lesssim& \|f\|_{L^2},
    \end{align*}
    so we obtain \eqref{eq: improved K0 bound}.
\end{proof}
Using the above proposition, it suffices to show that, for a fixed $\lambda$ in the interior of $B^*$ and any $t\geq1$, there exists $G_{t\lambda}\in L^2(X)$ such that
\begin{align}\label{eq: lower bound for K0 when p small}
    \| K_{t\lambda}^0 G_{t\lambda}\|_{L^p(X)}\gtrsim t^{r\delta(n/r,p)}\|G_{t\lambda}\|_{L^2(X)},\qquad 2\leq p\leq p_c.
\end{align}
We construct $G_{t\lambda}$ to be an approximate Gaussian beam, which concentrates near $O(t^{-1/2})$ neighborhoods of a maximal flat. Let $\chi\in C_c^\infty(G)$ be a bump function. Without loss of generality, we assume that $\Gamma$ is torsion-free so we can assume that $\gamma\supp\chi\cap\supp\chi\neq \emptyset$ for $\gamma\in\Gamma$ implies that $\gamma=e$. With the Iwasawa decomposition $G=NAK$, given $x\in G$, we let $N(x)\in \fn$ and $A(x)\in\fa$ be such that $x\in e^{N(x)} e^{A(x)}K$. Let $b_0\in C_c^\infty(\BR)$ be as before, an even bump function that is 1 on $[-1,1]$ and vanishes outside $[-2,2]$. We let
\[
G_{t\lambda}(x)=t^{\frac{n-r}{4}}\chi(x)b_0\Big(t^{1/2}\|N(x)\|\Big)b_0\Big(\|A(x)\|\Big) e^{\langle\rho+it\lambda,A(x)\rangle},\qquad x\in G.
\]
By our assumption on $\supp \chi$, we shall think of $G_{t\lambda}\in C^\infty(X)$ and it is immediate that 
\begin{align}\label{eq: Lp norm of gaussian beam}
    \|G_{t\lambda}\|_{L^p(X)} \asymp t^{\frac{n-r}{4}-\frac{n-r}{2p}},\qquad p\geq1.
\end{align}

\begin{lemma}
    Let $K_{t\lambda}^0$ and $G_{t\lambda}$ be as above. Then we have
    \[
    |\langle K_{t\lambda}^0 G_{t\lambda},G_{t\lambda}\rangle|=\left| \int_X K_{t\lambda}^0G_{t\lambda}(x) \overline{G_{t\lambda}(x)} dx \right|\gtrsim 1.
    \]
\end{lemma}
\begin{proof}
    By our assumption on $\Gamma$ and $\chi$, it is clear that
    \begin{align*}
        \int_X K_{t\lambda}^0G_{t\lambda}(x) \overline{G_{t\lambda}(x)} dx=\int_G\int_G \overline{G_{t\lambda}(x)} G_{t\lambda}(y) k^0_{t\lambda}(x^{-1}y) dx dy.
    \end{align*}
    Using the inverse Harish-Chandra transform for $k^0_{t\lambda}$ and applying \cite[Chapter IV, Lemma 4.4]{helgason1984groups} to the spherical function, we have
    \begin{align*}
        &\langle K_{t\lambda}^0 G_{t\lambda},G_{t\lambda}\rangle=\frac{1}{|W|}\int_{\fa^*} h^0_{t\lambda}(s)\left(\int_G\int_G \overline{G_{t\lambda}(x)} G_{t\lambda}(y) \varphi_s(x^{-1}y)dxdy \right) d\mu(s)\\
        =&\frac{1}{|W|}\int_{\fa^*}\int_K h^0_{t\lambda}(s)\left(\int_G\int_G \overline{G_{t\lambda}(x)} G_{t\lambda}(y) e^{\langle \rho-is,A(kx)\rangle} e^{\langle \rho+is,A(ky)\rangle}dxdy \right) dkd\mu(s)\\
        =&\frac{1}{|W|}\int_{\fa^*}\int_K h^0_{t\lambda}(s)\left|I(t\lambda;s,k)\right|^2 dkd\mu(s).
    \end{align*}
    Here $d\mu$ is the Plancherel measure, and $I(t\lambda;s,k)$ is the integral
    \[
    \int_G \overline{G_{t\lambda}(x)} e^{\langle \rho-is,A(kx)\rangle} dx = t^{\frac{n-r}{4}}\int_G\chi(x)b_0(t^{1/2}\|N(x)\|)b_0(\|A(x)\|) e^{\langle\rho-it\lambda,A(x)\rangle}e^{\langle \rho-is,A(kx)\rangle}  dx.
    \]
    By considering the Iwasawa decomposition and combining the amplitudes to define a new cutoff function $\widetilde\chi(\cdot,\cdot;k)\in C_c^\infty(\fn\times\fa)$ with parameter $k\in K$, we can write $I(t\lambda;s,k)$ as the integral
    \begin{align}\label{eq: no osillation}
         t^{\frac{n-r}{4}}\int_{\fn\times\fa}\widetilde\chi(N,H;k) b_0(t^{1/2}\|N\|)b_0(\|H\|) e^{-i\langle s,A(ke^N e^H)\rangle-i\langle t\lambda,H\rangle} dN dH,
    \end{align}
    with $dN$ and $dH$ Lebesgue measures on $\fn$ and $\fa$. We choose a constant $\epsilon>0$. Suppose that the parameters $s,k$ in $I(t\lambda;s,k)$ satisfy that $\|s+t\lambda\|<\epsilon$ and $k$ is in an $\epsilon t^{-1/2}$ neighborhood of $M$ in $K$. Using Taylor's theorem and the fact that $M$ is contained in the critical set for the function $k\mapsto \langle\lambda, A(k e^H)\rangle$, we get
    \[
    \langle \lambda,A(ke^N e^H)\rangle=\langle\lambda, H\rangle+O(\epsilon t^{-1}).
    \]
    Therefore, the phase function in \eqref{eq: no osillation} becomes
    \[
    \langle s,A(ke^N e^H)\rangle+\langle t\lambda,H\rangle= \langle t\lambda, H-A(ke^N e^H)\rangle+O(\epsilon)=O(\epsilon),
    \]
    so the integral \eqref{eq: no osillation} is $\gtrsim t^{-(n-r)/4}$ by choosing $\epsilon$ sufficiently small. Using the fact that $$\int_{\|s+t\lambda\|\leq \epsilon} d\mu(s)\gtrsim_\epsilon t^{n-r}, $$
    and combining with the above, we obtain
    \begin{align*}
        \langle K_{t\lambda}^0 G_{t\lambda},G_{t\lambda}\rangle&=\frac{1}{|W|}\int_{\fa^*}\int_K h^0_{t\lambda}(s)\left|I(t\lambda;s,k)\right|^2 dkd\mu(s)\\
        &\geq \frac{1}{|W|} \int_{\|s+t\lambda\|\leq \epsilon}\int_{\mathrm{dist}(k,M)\leq \epsilon t^{-1/2}} h^0_{t\lambda}(s)\left|I(t\lambda;s,k)\right|^2 dkd\mu(s)\\
        &\gtrsim_\epsilon t^{n-r} t^{-(n-r)/2} ( t^{-(n-r)/4})^2 = 1.
    \end{align*}\end{proof}

Given $p\geq2$, let $p'$ be the dual exponent that is $\frac{1}{p}+\frac{1}{p'}=1$. Combining the above lemma with H\"older's inequality, we obtain
\begin{align*}
    \|K_{t\lambda}^0 G_{t\lambda}\|_{L^p(X)}\| G_{t\lambda}\|_{L^{p'}(X)}\geq  |\langle K_{t\lambda}^0 G_{t\lambda},G_{t\lambda}\rangle|\gtrsim 1.
\end{align*}
Using \eqref{eq: Lp norm of gaussian beam}, we obtain
\begin{align*}
     \|K_{t\lambda}^0 G_{t\lambda}\|_{L^p(X)}&\gtrsim  \| G_{t\lambda}\|_{L^{p'}(X)}^{-1} \gtrsim t^{\frac{n-r}{4}-\frac{n-r}{2p}}\asymp t^{\frac{n-r}{4}-\frac{n-r}{2p}}\| G_{t\lambda}\|_{L^{2}(X)},
\end{align*}
so we conclude \eqref{eq: lower bound for K0 when p small}.

\subsection{Generalization to semisimple groups}\label{sec: ss generalization}

Finally, we sketch how to modify our proof to get the correct bound in semisimple cases. In this section we follow the same notation as the previous sections, but we assume that $G$ is only semisimple and $S=S_1\times S_2\times\cdots\times S_m$. Here $S_j = G_j/K_j$, $1\leq j\leq m$, are irreducible factors, and each $G_j$ is simple with a maximal compact subgroup $K_j$. Let $n_j = \dim S_j$ and $r_j=\rank S_j$. Let $\Gamma\subset G$ be a cocompact lattice; we consider the locally symmetric space $X=\Gamma\backslash S$.

For each $j$, we let $k_j(\cdot;z)$ be the test function with complex parameter $z\in\BC$ constructed in Section \ref{sec: complex family of projectors}. Let $k(\cdot;z)=\prod_{j} k_j(\cdot;\frac{r_j}{n_j-r_j}z)$ and let $T_z$ be the operator associated to the automorphic kernel
\[
    T_z(x,y)=\sum_{\gamma\in\Gamma} k(x^{-1}\gamma y;z).
\]
It may be seen that the same proof for Proposition \ref{prop 6.4} gives
\[
\|T_z\|_{L^2\to L^2}\lesssim 1,\qquad\text{when }\Re z=-1,
\]
and
\[
\|T_z\|_{L^1\to L^\infty}\lesssim \prod_{j=1}^m \max\left\{  t^{(n_j-r_j)(1-\frac{r_j}{n_j-r_j}\Re z)}, t^{\frac{n_j-r_j}{2}}\right\}
\]
for any $\Re z\geq0$. Fix some $p\geq2$ and let $p'$ be its dual exponent. Then Stein's complex interpolation for the holomorphic family of operators $\{T_z\}$ in the vertical strip $-1\leq\Re z\leq \frac{2}{p-2}$ gives
\begin{align*}
    \|T_0\|_{L^{p'}\to L^p} &\lesssim \left(\prod_{j=1}^m \max\left\{  t^{(n_j-r_j)(1-\frac{r_j}{n_j-r_j}\frac{2}{p-2})}, t^{\frac{n_j-r_j}{2}}\right\}\right)^{\frac{p-2}{p}}\\
    &=\prod_{j=1}^m t^{\max\left\{ n_j-r_j-\frac{2n_j}{p},\frac{n_j-r_j}{2}-\frac{n_j-r_j}{p}\right\}}\\
    &=\prod_{j=1}^m t^{2r_j\delta(n_j/r_j,p)}.
\end{align*}
The rest of the argument is the same as the irreducible case.

\section{Fourier restriction}

In this section, we prove Theorem \ref{thm: Fourier restriction} or equivalently Theorem \ref{thm: Fourier extension}. 
We define, for $\Lambda\in\fa^+$ and $X\in\fp$, a generalized Bessel function to be
\[
J_{\Lambda}(X) = \int_K e^{i\langle X,\Ad k(\Lambda)\rangle} dk.
\]
Given $f\in C_c^\infty(\fp)$ we define its convolution with $J_\Lambda$ to be
    \[f*J_{\Lambda}(X)=\int_\fp f(X-Y)J_{\Lambda}(Y) dY.\]

By a standard $TT^*$ argument, it is equivalent to prove the following theorem.
\begin{theorem}\label{thm: convolution}
    Let $\Lambda\in\fa^+$ be a regular point. Let $p_c=\frac{2(n+r)}{n-r}$ and $ p_c' = \frac{2(n+r)}{n+3r}$. Then for any $f\in C_c^\infty(\fp)$ we have
    \[
    \|f*J_{\Lambda}\|_{p_c} \lesssim \|f\|_{p_c'}.
    \]
\end{theorem}

To prove this, first we fix a smooth function $\phi$ on $\fa^+$, with compact support, contained in a small neighborhood of the point $\Lambda$, and which takes the value 1 at $\Lambda$. We extend $\phi$ to a $K$-invariant smooth function on $\fp$, denoted again by $\phi$.  Since $\phi$ is identically 1 on the support of $d\mu_{\cO_\Lambda}$ we have ${\phi}^\vee * J_{\Lambda} = J_{\Lambda}$.

Similar to the previous proof, we use Stein's interpolation and now introduce a holomorphic family of operators. 
Let $b_0\in C_c^\infty(\BR)$ be an even cutoff function that is equal to 1 on $[-1,1]$ and zero outside $[-2,2]$, and let $b_1(x) = b_0(x)-b_0(2x)$. For any $H\in\fa$, we have
\[
1 = b_0\big(\omega(H)\big) + \sum_{m=1}^{\infty}  b_1\big(2^{-m}\omega(H)\big).
\]
For a complex variable $z\in\BC$ and a parameter $t\geq1$, we define, for $H'\in\fa$ and $k\in K$,
\begin{multline*}
    \mu_t(\Ad k(H');z) =(2^{z}-2^{-\frac{r}{n-r}})\\
    \sum_{w\in W}\int_\fa b_0(t^{-1}\|H\|)\left(b_0\big(\omega(H)\big) + \sum_{m=1}^{\infty} 2^{mz}  b_1\big(2^{-m}\omega(H)\big)\right)e^{i\langle H,H'-w\Lambda\rangle} dH.
\end{multline*}
Here the Lebesgue measure $dH$ on $\fa$ is normalized so that the polar coordinates integration formula
\[ \int_\fp f(X) dX = \int_K \int_{\fa^+}f(\Ad k (H))\omega(H)dH dk\]
holds for any $f\in C_c^\infty(\fp)$.
For each fixed $t\geq1$, we define a holomorphic family of operators $\{T_{z,t}\}_{z\in\BC}$ by the convolution $$T_{z,t}: f\mapsto f*(\phi  \mu_t(\cdot;z))^\vee $$
where $f\in C_c^\infty(\fp)$.

\begin{lemma}\label{lem: dominated convergence}
    When $z=0$, for any $f\in C_c^\infty(\fp)$ and any $X\in\fp$, we have
    \[
    \lim_{t\to\infty} T_{0,t}f(X) =C f*J_\Lambda(X),
    \]
    where $C$ is a nonzero constant.
\end{lemma}
\begin{proof}
    It is clear that for $H'\in\fa$
    \[
    \mu_t(H';0)=(1-2^{-\frac{r}{n-r}})
    \sum_{w\in W}\int_\fa b_0(t^{-1}\|H\|)e^{i\langle H,H'-w\Lambda\rangle} dH,
    \]
    so using the $K$-invariance for $\phi$ and $\mu_t$ and the polar coordinates, we have for $X\in \fp$
    \begin{align*}
        &(\phi  \mu_t(\cdot;0))^\vee( X) = \int_\fp \phi(Y) \mu_t(Y;0) e^{i\langle X,Y\rangle} dY\\
        &=\int_{\fa^+}\int_K \phi(H')\mu_t(H';0)e^{i\langle X,\Ad k(H')\rangle}\omega(H') dk dH'\\
        &=(1-2^{-\frac{r}{n-r}})
    \sum_{w\in W}\int_{\fa^+}\int_\fa \omega(H')\phi(H')b_0(t^{-1}\|H\|)e^{i\langle H,H'-w\Lambda\rangle} \left(\int_K e^{i\langle X,\Ad k(H')\rangle} dk\right) dHdH'.
    \end{align*}
    Taking the limit as $t\to \infty$, the integral for $H$ above gives a delta distribution (up to a constant multiple) at $w\Lambda$. Therefore, there exists a nonzero constant $C$ such that
    \[
    \lim_{t\to\infty}(\phi  \mu_t(\cdot;0))^\vee( X) = C \omega(\Lambda)\phi(\Lambda)J_{\Lambda}(X).
    \]
    Therefore, for any $t\geq1$ and $X\in\fp$, by the dominated convergence theorem, we have
    \begin{align*}
        \lim_{t\to \infty}T_{0,t}f(X)=\lim_{t\to\infty}\int_\fp f(X-Y)(\phi  \mu_t(\cdot;0))^\vee(Y) dY = \int_\fp f(X-Y)\lim_{t\to\infty}(\phi  \mu_t(\cdot;0))^\vee(Y) dY,
    \end{align*}
     which completes the proof.
\end{proof}

\begin{proposition}\label{prop: complex bound for restriction}
    For any $t\geq1$, when $\Re z = \frac{1}{2}$ we have $\|T_{z,t}\|_{L^1\to L^\infty}\lesssim1$, and when $\Re z = -\frac{r}{n-r}$ we have $\|T_{z,t}\|_{L^2\to L^2}\lesssim1$, where both of the implied constants are independent of $t$.
\end{proposition}

\begin{proof}[Proof of Theorem \ref{thm: convolution} from Proposition \ref{prop: complex bound for restriction}]
    From Stein's complex interpolation theorem for the holomorphic family of operators $T_{z,t}$ in the vertical strip $-\frac{r}{n-r}\leq\Re z\leq \frac{1}{2}$, we have $\|T_{0,t}\|_{p_c'\to p_c}\lesssim 1$ uniformly for all $t\geq1$. Let $C$ be the constant from Lemma \ref{lem: dominated convergence}. Applying Fatou's lemma to $|T_{0,t}f|^{p_c}$ we get
    \begin{align*}
        \| f*J_\Lambda\|_{p_c}\leq |C|^{-1} \liminf_{t\to\infty} \| T_{0,t}f\|_{p_c}\lesssim\|f\|_{p_c'},
    \end{align*}
    so we prove Theorem \ref{thm: convolution}.
\end{proof}

\begin{proof}[Proof of Proposition \ref{prop: complex bound for restriction}]
    By Young's inequality, $\|T_{z,t}\|_{L^1\to L^\infty}$ is dominated by the sup-norm of $(\phi  \mu_t(\cdot;z))^\vee $. Moreover, by the $K$-invariance, it suffices to prove, for $X\in\fa$ and $\Re z =\frac{1}{2}$, the boundedness of
\begin{align*}
    &(\phi  \mu_t(\cdot;z))^\vee(X)\\
    =&\int_K \int_{\fa^+} \phi(H')\mu_t(H';z) e^{i\langle X,\Ad k(H')\rangle} \omega(H') dH' dk\\
    =&(2^{z}-2^{-\frac{r}{n-r}})\sum_{w\in W}\int_\fa e^{-i\langle H,w\Lambda\rangle}b_0(t^{-1}\|H\|)\left(b_0\big(\omega(H)\big) + \sum_{m=1}^{\infty} 2^{mz}  b_1\big(2^{-m}\omega(H)\big)\right)\\
    &\left(\int_{\fa^+}\omega(H') \phi(H')e^{i\langle H,H'\rangle}J_{H'}(X) dH'\right)dH.
\end{align*}
Applying Theorem \ref{thm: Weyl Harish-Chandra symbol 2} to $\omega(H')\phi(H')J_{H'}(X)$, there exist coefficients $b_{w}(X,H')$ which are smooth in $H'$ and satisfy the derivative bounds \eqref{eq:symbol 2}, so that
\[\omega(H')\phi(H')J_{H'}(X)=\sum_{w\in W}e^{i\langle H',wX\rangle}b_{w}(X,H').\]
Therefore,
\begin{align}\label{eq: nonstationary phase for J integral}\begin{split}
    \int_{\fa^+}\omega(H') \phi(H')e^{i\langle H,H'\rangle}J_{H'}(X) dH' =\sum_{w\in W}\int_{\fa^+} b_{w}(X,H') e^{i\langle H',wX+H\rangle} dH'\\
    \lesssim_N \prod_{\alpha\in\Delta^+}(1+|\alpha(X)|)^{-m_\alpha/2} \sum_{w\in W} (1+\|H+wX\|)^{-N}.
\end{split}\end{align}
Moreover, for any $H\in\fa$, there are at most $O(1)$ many $m$ so that $b_1\big(2^{-m}\omega(H)\big)\neq0$, so for $\Re z=\frac{1}{2}$
\begin{align*}
    b_0\big(\omega(H)\big) + \sum_{m=1}^{\infty} 2^{mz}  b_1\big(2^{-m}\omega(H)\big) \lesssim 1+ |\omega(H)|^{1/2}\lesssim \prod_{\alpha\in\Delta^+}(1+|\alpha(H)|)^{m_\alpha/2}.
\end{align*}
 For each positive root $\alpha\in\Delta^+$, the triangle inequality gives
    \begin{align*}
        1+|\alpha(H)|\lesssim (1+|\alpha(wX)|)(1+\|H+wX\|),
    \end{align*}
so taking the product gives, for any $w\in W$,
\begin{align*}
    b_0\big(\omega(H)\big) + \sum_{m=1}^{\infty} 2^{mz}  b_1\big(2^{-m}\omega(H)\big) \lesssim \prod_{\alpha\in\Delta^+}(1+|\alpha(H)|)^{m_\alpha/2}\\
    \lesssim(1+\|H+wX\|)^{\frac{n-r}{2}} \prod_{\alpha\in\Delta^+}(1+|\alpha(X)|)^{m_\alpha/2}.
\end{align*}
Combining this with \eqref{eq: nonstationary phase for J integral} yields that
\begin{align*}
    (\phi  \mu_t(\cdot;z))^\vee(X)\lesssim_N\sum_{w\in W} \int_\fa (1+\|H+wX\|)^{-N} dH \lesssim1.
\end{align*}

When $\Re z=-\frac{r}{n-r}$, to prove the $L^2\to L^2$ boundedness, by the Plancherel theorem, it suffices to show that
\begin{align}\label{eq: phi J bound}
    \sup_{H'\in\fa}|\mu_t(H';z)|\lesssim 1.
\end{align}
We have, for $z=-\frac{r}{n-r}+iy$ with $y\in\BR$,
\begin{align*}
    \mu_t(H';z) =&2^{-\frac{r}{n-r}}(2^{iy}-1)\\
    &\sum_{w\in W}\int_\fa b_0(t^{-1}\|H\|)\left(b_0\big(\omega(H)\big) + \sum_{m=1}^{\infty} 2^{mz}  b_1\big(2^{-m}\omega(H)\big)\right)e^{i\langle H,H'-w\Lambda\rangle} dH.
\end{align*}
Therefore, \eqref{eq: phi J bound} follows from the boundedness of
\begin{align*}
    &(2^{iy}-1)\int_\fa b_0(t^{-1}\|H\|)\left(b_0\big(\omega(H)\big) + \sum_{m=1}^{\infty} 2^{mz}  b_1\big(2^{-m}\omega(H)\big)\right)e^{i\langle H,H'\rangle} dH\\
    =&t^r(2^{iy}-1)\int_\fa b_0(\|H\|)\left(b_0\big(t^{n-r}\omega(H)\big) + \sum_{m=1}^{\infty} 2^{mz}  b_1\big(t^{n-r}2^{-m}\omega(H)\big)\right)e^{i\langle H,tH'\rangle} dH
\end{align*}
uniformly for $H'\in\fa$. This follows directly from Proposition \ref{prop: symbol estimates}.
\end{proof}


\bibliographystyle{alpha} 
\bibliography{refs.bib} 

\end{document}